\documentclass[11pt,a4paper]{article}

\usepackage[british]{babel}

\usepackage[T1]{fontenc}

\usepackage[utf8]{inputenc}

\usepackage{csquotes}

\usepackage[backend=biber, maxbibnames=10, maxcitenames=4, maxalphanames=4, bibencoding=utf8, style=alphabetic, isbn=false, url=false, doi=true]{biblatex}
\renewbibmacro{in:}{}
\renewbibmacro*{doi+eprint+url}{%
	\printfield{doi}%
	\newunit\newblock%
	\iftoggle{bbx:eprint}{%
		\usebibmacro{eprint}%
	}{}%
	\newunit\newblock%
	\iffieldundef{doi}{%
		\usebibmacro{url+urldate}}%
	{}%
}

\usepackage{setspace}

\usepackage{mathpazo}

\usepackage[indentafter]{titlesec}
\titleformat{name=\section}{}{\thetitle.}{0.8em}{\centering\scshape}
\titleformat{name=\subsection}[runin]{}{\thetitle.}{0.5em}{\bfseries}[.]
\titleformat{name=\subsubsection}[runin]{}{\thetitle.}{0.5em}{\itshape}[.]
\titleformat{name=\subparagraph,numberless}[runin]{}{}{0em}{}[.]
\titlespacing{\subparagraph}{0em}{0em}{0.5em}

\usepackage{url}

\usepackage{mathtools}

\DeclarePairedDelimiter\abs{\lvert}{\rvert}%
\DeclarePairedDelimiter\norm{\lVert}{\rVert}%

\makeatletter
\let\oldabs\abs
\def\abs{\@ifstar{\oldabs}{\oldabs*}}
\let\oldnorm\norm
\def\norm{\@ifstar{\oldnorm}{\oldnorm*}}
\makeatother

\usepackage{amssymb}

\usepackage{amsthm}

\usepackage{mathrsfs}

\usepackage{dsfont}

\usepackage{enumitem}
\setlist[enumerate]{noitemsep, partopsep=0pt, topsep=0pt, parsep=0pt, itemsep=0pt}
\setlist[itemize]{noitemsep, partopsep=0pt, topsep=0pt, parsep=0pt, itemsep=0pt}

\usepackage{interval}

\intervalconfig{soft open fences}

\usepackage[normalem]{ulem}

\usepackage[stretch=10]{microtype}

\usepackage[affil-it, noblocks]{authblk}

\usepackage{xcolor}

\usepackage[hidelinks,draft=false]{hyperref}
\hypersetup{
  colorlinks   = true, 
  urlcolor     = blue, 
  linkcolor    = blue, 
  citecolor   = red 
}
\usepackage{breakurl}

\usepackage[nameinlink,noabbrev,capitalize]{cleveref}
\crefname{equation}{}{}

\usepackage[plain]{fullpage}

\newlist{theoenum}{enumerate}{1} 
\setlist[theoenum]{label=\normalfont(\roman*), ref=\theproposition~\normalfont(\roman*), noitemsep, partopsep=0pt, topsep=0pt, parsep=0pt, itemsep=0pt}
\crefalias{theoenumi}{theorem}

\usepackage{tocloft}

\usepackage{titlefoot}

\theoremstyle{plain}
\newtheorem{theorem}{Theorem}
\newtheorem*{theorem*}{Theorem}
\newtheorem{proposition}[theorem]{Proposition}

\newtheorem{lemma}[theorem]{Lemma}
\newtheorem*{lemma*}{Lemma}

\theoremstyle{definition}
\newtheorem{definition}[theorem]{Definition}

\theoremstyle{remark}

\newtheorem{remark}[theorem]{Remark}
\newtheorem*{remark*}{Remark}			

\newcommand*\diff{\mathop{}\!\mathrm{d}}
\newcommand*\Diff{\mathop{}\!\mathrm{D}}
\newcommand{\N}{\mathbb{N}}
\newcommand{\R}{\mathbb{R}}

\newcommand{\D}{\mathcal{D}}

\renewcommand{\phi}{\varphi}
\renewcommand{\L}{\mathrm{L}}
\renewcommand{\exp}{\mathrm{exp}}

\DeclareMathOperator{\T}{\mathrm{T}}

\let\originalleft\left
\let\originalright\right
\renewcommand{\left}{\mathopen{}\mathclose\bgroup\originalleft}
\renewcommand{\right}{\aftergroup\egroup\originalright}

\renewcommand{\H}{\mathds{H}}

\date{\today}
\title{\textbf{\uppercase{\large{Optimal Transport on the sub-Lorentzian Heisenberg group}}}}
\author[2]{Samuël Borza\thanks{\href{mailto:samuel.borza@univie.ac.at}{samuel.borza@univie.ac.at}}}
\author[1]{Wilhelm Klingenberg\thanks{\href{mailto: wilhelm.klingenberg@dur.ac.uk}{wilhelm.klingenberg@dur.ac.uk }}}
\author[1]{Patrick Wood\thanks{\href{mailto:patrick.a.wood@durham.ac.uk}{patrick.a.wood@durham.ac.uk}}}
\affil[1]{Department of Mathematics, Durham University}
\affil[2]{Faculty of Mathematics, University of Vienna}					

\begin{document}

 \maketitle							

\providecommand{\keywords}[1]
{
	\textbf{\textit{Keywords---}} #1
}

\providecommand{\msc}[1]
{
	\textbf{\textit{MSC (2020)---}} #1
}

\begin{abstract}
 	 We investigate the synthetic metric spacetime structure of the sub-Lorentzian Heisenberg group and we study the optimal transport problem in this space. The sub-Lorentzian version of Brenier's theorem is established in this setting. Finally, we provide examples of optimal transport maps and derive a sub-Lorentzian Monge-Ampère equation.
\end{abstract}


\msc{53C50, 53C17, 49Q22}				

 \begin{spacing}{0.9}
 \setlength{\cftbeforesecskip}{0pt}
 \vspace{-0.4cm}\tableofcontents
 \end{spacing}

\section{Introduction}

Sub-Lorentzian geometry is to Lorentzian geometry what sub-Riemannian geometry is to Riemannian geometry. A sub-Lorentzian space is a smooth manifold equipped with a Lorentzian metric defined at every point on a subspace of the tangent space, known as the horizontal space of admissible directions. There is an extensive body of literature and active research in sub-Riemannian geometry; see comprehensive references such as \cite{Montgomery2002, comprehensive2020}. However, relatively little research has been conducted on sub-Lorentzian geometry. The idea of such a theory was envisaged in Strichartz’s seminal paper on sub-Riemannian geometry \cite{Strichartz1986}, but it was really initiated through Grochowski’s work, whose first publication on the topic seems to be \cite{Grochowski2002}.

The Heisenberg group is the prototypical and simplest example of a geometric space with non-holonomic constraints. It provides significant insight into the theory and has been studied extensively when equipped with a sub-Riemannian metric—see again \cite{comprehensive2020}, as well as \cite{Capogna2007}. It is therefore natural to consider the Heisenberg group equipped with a sub-Lorentzian metric, and important contributions have been made in \cite{Grochowski2004, Grochowski2006, sachkov2022sublorentzian}.

In the past few decades, the study of length spaces and the analysis of metric measure spaces have been very successful in developing an abstract, synthetic, and non-smooth theory encompassing (sub)-Riemannian and (sub)-Finslerian theories, as well as their limits in the sense of Gromov-Hausdorff. Recently, a theory of Lorentzian length spaces and metric spacetimes was introduced in \cite{kuzingersamann2018} with the goal of developing an analogous synthetic setting for Lorentzian geometry. While in metric geometry the central object is the distance function, it is the time-separation, i.e. the maximum (proper) time as measured by a clock along any line between two given points, that is taken as the embodiment of the fundamental structural feature of a spacetime. In addition to being interesting from a mathematical point of view, this approach is also motivated in physics by the need to handle spacetimes with very low regularity, both in general relativity and in quantum gravity.

Sub-Lorentzian geometry is also pertinent from a physical point of view, but it has not yet been studied within the non-smooth framework of Lorentzian length spaces. One of the contributions of the present work is to show (see \cref{subsection:HeisenbergMetricSpacetime}) that the sub-Lorentzian Heisenberg group is indeed a Lorentzian metric spacetime in the sense of \cite{kuzingersamann2018}, with regularity properties and its place within the ``causal ladder'' identified in \cref{prop:HisGloballyHyperbolic}. This is achieved after reviewing, in \cref{section:subLgeometry}, the sub-Lorentzian structure of the Heisenberg group ``from the Hamiltonian point of view'', according its the optimal control formulation of \cite{sachkov2022sublorentzian}.

An important turning point in the recent development of metric geometry was the discovery of the central role that optimal transport can play, see \cite{Villani2009} and \cite{Santambrogio2015} for a detailed overview and futher references. From curvature-bound conditions to isoperimetric problems, and even interpolation inequalities, optimal transport has become a key tool in modern geometric analysis. It is therefore not surprising that the foundation for a Lorentzian version of optimal transport theory, compatible with the framework from \cite{kuzingersamann2018}, was soon laid down in \cite{Mondino2023, CavallettiMondino2023, CavallettiMondino2022}.

In modern language, Monge's original problem, which he introduced and studied in 1781, can be formulated as follows. Given $p \geq 1$ and $\mu, \nu \in \mathcal{P}(\R^n)$ that are assumed, for simplicity, to be absolutely continuous with respect to the Lebesgue measure and with compact support, find a map that minimises the functional
\[
M(T) = \frac{1}{p} \int_{\R^n} \diff(x, T(x))^p \diff \mu(x)
\]
among all the measurable maps $T : \R^n \to \R^n$ that transport the measure $\mu$ to $\nu$, i.e. satisfying $T_\sharp \mu = \nu$. Monge was motivated by an engineering problem where an amount of material extracted from the earth or a mine, the ``déblai'' modelled by $\mu$, is to be placed into a new construction, the ``remblai'' modelled by $\nu$, in the most efficient way. This problem with $p = 1$ is the one that he originally considered and roughly corresponds to minimising the total distance that each particle of material will travel. When $p = 2$, it is more of an energy minimisation that is at play. Because of convexity issues, it turned out that the case $p = 1$ is actually much trickier than for $p > 1$. The latter case was resolved by Brenier in his celebrated paper \cite{Brenier1991}, and the result is now known as Brenier's theorem.
For another clear explanation, see also \cite[Theorem 1.17]{Santambrogio2015}, and we refer the reader to \cite[Chapter 3]{Santambrogio2015} for the former case $p = 1$, which we will not address in the present work. From there, Brenier's theorem was generalised to other background spaces: in Riemannian geometry in \cite{mccann2001}, in the sub-Riemannian Heisenberg group in \cite{AmbrosioRigot2004}, and for large classes of sub-Riemannian manifolds in \cite{Figalli2010}. In any such result, the optimal transport map $T$ is proven to be unique, and expressed as the geodesic exponential map of the differential of a convex function, where "convexity" is understood in a suitable sense dependent on the cost $d^p/p$.

 We briefly explain the Lorentzian version of Monge's optimal transport problem in order to state our result. Given $p \in \linterval{0}{1}$, a Lorentzian length space $X$ with a time-separation function $\tau$, and $\mu, \nu \in \mathcal{P}(X)$, the (forward) Lorentz-Monge problem consists in maximising the functional
\[
M^+(T) = \frac{1}{p}\int_{X} \tau(x, T(x))^p \diff \mu(x)
\]
among all the measurable maps $T: X \to X$ such that $T_\sharp \mu = \nu$, and such that $\tau(x, T(x)) > 0$ for $\mu$-almost every $x \in X$. A map maximising this functional is called a (forward) optimal transport map from $\mu$ to $\nu$. This latter condition is necessary to ensure that transport occurs forward in time and in a chronological manner, i.e. $T(x)$ lies in the future of $x$, and the journey from $x$ to $T(x)$ can be completed at a speed lower than that of light. The cosmological interpretation here is that we wish to transport a spacetime gas with configuration $\mu$ chronologically to a future new configuration $\nu$ in the most optimal way, that is to say, while maximising proper time. Analogously, we can study the Lorentz-Monge transport problem that is backward in time, i.e. maximise the function
\[
M^-(T) = \frac{1}{p}\int_{X} \tau(T(x), x)^p \diff \mu(x)
\]
among all the measurable maps $T: X \to X$ such that $T_\sharp \mu = \nu$, and such that $\tau(T(x), x) > 0$ for $\mu$-almost every $x \in X$.

In this paper, the Lorentzian length space is the sub-Lorentzian Heisenberg group, which we denote by $\H$, and we assume that $p \in \ointerval{0}{1}$. In the following statement, the Hamiltonian $H$ is given in \cref{eq:HamiltonianH}, while the notion of a $c_p$-concave map will be explored in detail in \cref{section:OTinSLH}, with the notation $c_p := \tau^p/p$. For now, it suffices to view a $c_p$-concave map $\phi : A_1 \subseteq \H \to \R$ as a locally semiconcave function, which is, in particular, differentiable almost everywhere. We denote by $\H_{\ll}^2$ the set of pairs of points $(q, q') \in \H^2$ such that $\tau(q, q') > 0$. The exponential map $\exp_q$ represents the sub-Lorentzian geodesic flow, and we will discuss it in detail in \cref{section:subLgeometry}.

\begin{theorem}[Brenier's theorem in the sub-Lorentzian Heisenberg group]\ \\
	Let $p \in \ointerval{0}{1}$, $\mu, \nu \in \mathcal{P}(\H)$ with compact support,  $\mu, \nu \ll \mathcal{L}^3$, and $\mathrm{supp}(\mu) \times \mathrm{supp}(\nu) \subseteq \H_{\ll}^2$. Then, the following holds.
	\begin{theoenum}
    \item There is a unique forward optimal transport map $T^{\mu \to \nu}$ from $\mu$ to $\nu$ and it is given by
    \begin{equation}
    \label{eq:MainT(q)}
	T^{\mu \to \nu}(q) = \exp_q\left(-{\Diff_{q} \phi}\bigg/{\left(\sqrt{2H(q, \Diff_q \phi)}\right)^{\tfrac{p - 2}{p-1}}}\right),
\end{equation}
where $\phi : A_1 \to \R$ is any map that is $c_p$-concave relative to a pair of sets $(A_1, A_2)$, with $\mu(A_1) = \nu(A_2) = 1$, $A_1$ being open, $A_2$ bounded, $\mathrm{supp}(\mu) \subseteq A_1$, and $A_1 \times A_2 \subseteq \H_\ll^2$.
	\item Conversely, if $A_1$ and $A_2$ are non-empty subsets of $\H$ such that $A_1 \times A_2 \subseteq \H_{\ll}^2$, if $A_1$ is open, and if $\phi : A_1 \to \R$ is a $c_p$-concave map relative to $(A_1, A_2)$ such that $\phi$ is differentiable $\mu$-almost everywhere, then the map $T$ given by \cref{eq:MainT(q)} is a forward optimal transport map from $\mu$ to $\nu$.
	\item Denoting by $T^{\nu \to \mu}$ the unique backward optimal transport map from $\nu$ to $\mu$, we get that
	\[
	T^{\mu \to \nu} \circ T^{\nu \to \mu} = \mathrm{Id} \ \ \ \text{ $\nu$-a.e. in $\H$, } \ \ \ \ \ T^{\nu \to \mu} \circ T^{\mu \to \nu} = \mathrm{Id} \ \ \ \text{ $\mu$-a.e. in $\H$}.
	\]
    \end{theoenum}
\end{theorem}

We prove this theorem through a series of results in \cref{subsection:brenierSLH}, removing the assumptions of absolute continuity and compactness of support whenever possible.

There are Lorentzian versions of Brenier's theorem in the literature, specifically in Lorentzian manifolds \cite{McCann2020} and in Finsler spacetimes \cite{BraunOhta2024}, and our result should be the first in the sub-Lorentzian setting. As will be made clear from the proof, a key fact used to establish the theorem is the regularity of the time-separation function in the sub-Lorentzian Heisenberg group. This was achieved in \cite[Theorem 6]{sachkov2022sublorentzian} and we recall this result in \cref{theorem:propertytau}.
The elements of the argument specific to the sub-Lorentzian setting are inspired by \cite{AmbrosioRigot2004} and \cite{Figalli2010}: we formulate a sub-Lorentzian Lagrange multiplier rule for timelike covectors in \cref{subsection:LagrangeMultiplier} and use it to derive the key geometric \cref{lemma:Geometric} and \cref{lemma:Geometric2}, which are well known in sub-Riemannian geometry. Finally, we deduce in \cref{subsection:MongeAmpere+examples} a Monge-Ampère type equation for the optimal transport map which is a classical consequence of Brenier's theorem, and we show some examples of optimal transport maps, including a study of the optimality of the right-translations.

We conclude this introduction with some open questions. We have not followed the approach in \cite[Section 6]{AmbrosioRigot2004} regarding the convergence of transport maps in the Lorentzian approximation of $\H$ to their sub-Lorentzian counterparts. This could be interesting and useful. The ideas presented in this paper are quite standard, but what prevents us from proving a Brenier's theorem for general sub-Lorentzian manifolds beyond the Heisenberg group is that we do not have enough knowledge on the regularity properties of the time-separation function. In sub-Riemannian geometry, the regularity properties of the distance function and the relationship with abnormal minimisers are well understood, as discussed in \cite[Chapter 11]{comprehensive2020}. A sub-Lorentzian version of these results is still lacking. As already mentioned, our result does not cover the case $p = 1$, for which different techniques should be employed, see \cite[Chapter 3]{Santambrogio2015} to have an idea in the Euclidean setting. A study of the so-called timelike curvature bounds $\mathsf{TCD}$ and $\mathsf{TMCP}$ will appear in a subsequent work.

\section*{Acknowledgements}

This project has received funding from the European Research Council (ERC) under the European Union’s Horizon 2020 research and innovation programme (grant agreement No. 945655). This research was funded in whole or in part by the Austrian Science Fund (FWF) [10.55776/EFP6]; and the Engineering and Physical Sciences Research Council (EPSRC) [EP/W523951/1]. 

We would like to thank Profs Yuri Sachkov and Séverine Rigot for their interest and helpful comments.		

\section{The sub-Lorentzian Heisenberg group}
\label{section:SLHeisenberg}

\subsection{Spacetime geometry of the sub-Lorentzian Heisenberg group}
\label{section:subLgeometry}

The Heisenberg group $\H$ is the connected and simply connected Lie group whose Lie algebra is graded, nilpotent, satisfying $\mathfrak{h} = \mathfrak{h}_1 \oplus \mathfrak{h}_2$ and such that $\mathfrak{h}_1$ is two-dimensional and generates $\mathfrak{h}$. In particular, $\mathds{H}$ is three-dimensional and fixing a basis $X, Y,$ and $Z = [X, Y]$ of $\mathfrak{h}$ induces global coordinates on $\mathds{H}$, called \emph{exponential coordinates}, through the map $(x, y, z) \mapsto \exp(x X + y Y + z Z)$. By the Campbell-Hausdorff formula, the law group in these coordinates writes as
\[
(x, y, z) \cdot (x', y', z') = \left(x + x', y + y', z + z' + \tfrac{1}{2}(x y' - x' y)\right).
\]
Using the left-translation $L_q : \H \to \H : q' \mapsto q \cdot q'$, we can define, with a slight abuse of notation, the left-invariant vector fields on $\mathds{H}$ whose values at the identity are $X, Y$ and $Z$. These will also be denoted by $X, Y$ and $Z$ respectively. Their expression in exponential coordinates is
\[
X = \partial_x - \frac{y}{2} \partial_z, \ Y = \partial_y + \frac{x}{2} \partial_z, \ \text{ and
 } \ Z = \partial_z.
\]
The distribution $\mathcal{D}$ spanned by $X$ and $Y$ is left-invariant and bracket generating, that is to say, $\mathrm{Lie}(\mathcal{D}_q) = \T_q(\mathds{H})$ for all $q \in \mathds{H}$ since the only non-zero Lie bracket is $[X, Y] = Z$.

The Heisenberg group $\H$ has been studied extensively when equipped with its natural sub-Riemannian structure, i.e. with a scalar product on $\mathcal{D}$ that turns $X$ and $Y$ into an orthonormal basis.
 In this work, we propose to further the study of the Heisenberg group endowed with a \emph{sub-Lorentzian} structure, in line with the previous works \cite{sachkov2022sublorentzian, Huang2012, Grochowski2006, Grochowski2004}.

Before addressing the Lorentzian optimal transport problem on this structure, we review its sub-Lorentzian geometry. The \textit{sub-Lorentzian metric} of $\H$ is the Lorentzian metric
$\langle \cdot, \cdot \rangle$ on $\D$, i.e. a non-degenerate, smooth, symmetric metric tensor on $\mathcal{D}$ that has index 1, defined by
\[
\langle u, v \rangle_p := u_2 v_2 - u_1 v_1,
\]
for all $u = u_1 X(q) + u_2 Y(q)$ and $v = v_1 X(q) + v_2 Y(q)$ in $\D_q$, and all $q \in \H$. In other words, it is the Lorentzian metric on $\D$ satisfying $\langle X, X \rangle = -1$ and $\langle Y, Y \rangle = 1$. For a given $q \in \H$, a vector $v \in \T_q(\H)$ is \emph{horizontal} if $v \in \D_q$. A horizontal vector $v \in \D_q$ is said to be timelike if $\langle v, v \rangle_q < 0$, null or lightlike if $\langle v, v \rangle_q = 0$ and $v \neq 0$, spacelike if $\langle v, v \rangle_q > 0$ or $v = 0$, and causal (or non-spacelike) if $\langle v, v \rangle_q \leq 0$ and $v \neq 0$.

The vector field $X$ is chosen as the \textit{time orientation}: for $q \in \H$, a horizontal vector $v \in \D_q$ is future-directed if $\langle v, X(p) \rangle_q < 0$ and past-directed if $\langle v, X(p) \rangle_q > 0$, adopting the ``mostly plus'' metric sign convention.

A horizontal curve is an absolutely continuous curve $\gamma : \interval{0}{T} \to \H$ such that $\dot{\gamma}(t) \in \D_{\gamma(t)}$ for almost every $t \in \interval{0}{T}$, i.e. if there exists a control $u = (u_1, u_2) : \interval{0}{T} \to \R^2$ in $\L^{\infty}(\interval{0}{T}, \R^2)$ such that
\begin{equation}
    \label{eq:controlledeq}
    \dot{\gamma}(t) = u_1(t) X(\gamma(t)) + u_2(t) Y(\gamma(t)), \ \ \text{for a.e. } t \in \interval{0}{T}.
\end{equation}
We will denote by $P_u^{t_0, t_1}(q_0)$ the associated flow, i.e. $P_u^{t_0, t_1}(q_0) := \gamma(t_1)$ where $\gamma$ solves \cref{eq:controlledeq} with initial condition $\gamma(t_0) = q_0$.
A horizontal curve $\gamma : \interval{0}{T} \to \H$ is timelike (resp. spacelike, lightlike, causal, future-directed, past-directed) if $\dot{\gamma}(t)$ is a timelike (resp. spacelike, lightlike, causal, future-directed, past-directed) horizontal vector for almost every $t \in \interval{0}{T}$. The length of a causal curve $\gamma : \interval{0}{T} \to \H$ controlled by $u \in \L^\infty(\interval{0}{T}, \R^2)$ is defined by
\begin{equation}
\label{eq:lorentzianlength}
	\L(\gamma) := \int_{0}^T |\langle \dot{\gamma}(t), \dot{\gamma}(t) \rangle_{\gamma(t)}|^{1/2} \diff t = \int_{0}^T \sqrt{u_1(t)^2 - u_2(t)^2} \diff t.
\end{equation}
For points $q_0, q \in \H$, the set $\Omega_{q_0q}$ denotes the set of future-directed causal curves $\gamma : \interval{0}{T} \to \H$ joining $q_0$ to $q$. A timelike curve $\gamma$ is parametrised by constant speed $v > 0$ if $\L(\gamma|_{\interval{s}{t}}) = v |s - t|$ for all $t, s \in \interval{0}{T}$, i.e. when $u_1(t)^2 - u_2(t)^2 = v^2$ for almost every $t \in \interval{0}{T}$. Any future-directed timelike curve can be reparametrised to have constant speed. By convention, we will also say that a lightlike curve has constant speed $v > 0$ if $u_1(t) = v$ for almost every $t \in \R$.

\begin{definition}    \label{timeseparation}
    The time-separation function of the sub-Lorentzian Heisenberg group $\H$ is the map $\tau : \H \times \H \to \interval{0}{\infty}$ defined as
    \begin{equation}
        \label{eq:tautimeseparation}
        \tau(q_0, q) :=
    \begin{cases}
        \ell(q_0, q) := \sup \left\{ \L(\gamma) \mid \gamma \in \Omega_{q_0q} \right\} & \text{ if } \Omega_{q_0q} \neq \emptyset \\
        0 & \text{ otherwise }
    \end{cases}.
    \end{equation}
\end{definition}

The causal future (resp. the chronological future) of a point $q_0 \in \H$ is the set $J^+(q_0)$ (resp. $I^+(q_0)$) of points $q \in \H$ for which there exists a future-directed causal (resp. a future-directed timelike) curve joining $q_0$ to $q$. The causal past $J^-(q_0)$ and chronological past $I^-(q_0)$ of a point $p \in \H$ are defined similarly with past-directed curves. We also set $\H_{\leq}^2 := \{(q_0, q) \in \H^2 \mid \ell(q_0, q) \geq 0 \}$ and $\H_{\ll}^2 := \{(q_0, q) \in \H^2 \mid \ell(q_0, q) > 0)\}$.

A maximising geodesic is a future-directed causal curve $\gamma : \interval{0}{T} \to \H$ joining $q_0$ to $q$ in $\H$ such that $\L(\gamma) = \tau(q_0, q)$. The search for a maximising geodesic $\gamma : \interval{0}{T} \to \H$ joining $q_0 \in \H$ to $q \in J^+(q_0)$ can be stated as the following optimal control problem:
\begin{equation}
    \tag{C}
    \label{nonspacelikeOCP}
    \begin{cases}
    \dot{\gamma}(t) = u_1(t) X(\gamma(t)) + u_2(t) Y(\gamma(t)) \\
    u(t) \in U := \left\{ (u_1, u_2) \in \R^2 \mid u_1 \geq \abs{u_2} \right\} \text{ for a.e. } t \in \interval{0}{T}\\
    \mathrm{L}(\gamma) \to \max \\
    \gamma(0) = q_0, \ \gamma(1) = q
    \end{cases}.
\end{equation}
Since the sub-Lorentzian structure of $\H$ is induced by left-invariant vector fields, the concepts defined above are invariant under left translations too. For instance, we have $\tau(q_0, q) = \tau(L_{q'}(q_0), L_{q'}(q))$,  $J^{\pm}(q_0) = J^{\pm}(L_q(q_0))$, etc.

For $\nu \in \R$, the sub-Lorentzian (controlled) Hamiltonian is the map $\mathscr{H} : U \times \mathrm{T}^*(\mathbb{H}) \times \mathbb{R} \to \mathbb{R}$ given by
\[
\mathscr{H}^\nu(u, \lambda) := u_1 h_X(\lambda) + u_2 h_Y(\lambda) - \nu \sqrt{u_1^2 - u_2^2}, \ \text{ for } u = (u_1, u_2) \in U \text{ and } \lambda \in \mathrm{T}^*(\mathbb{H}),
\]
where for a vector field $V$ on $\mathds{H}$, we have denoted by $h_V$ the function on $\T^*(\mathds{H})$ defined by $h_V(\lambda) := \langle \lambda, V(\pi(\lambda)) \rangle$. Note that $h_X$, $h_Y$, and $h_Z$ form a system of coordinates on each of the fibers of $\T^*(\mathds{H})$.

For a fixed $u \in U$ and $\nu \in \mathbb{R}$, the map $\lambda \mapsto \mathscr{H}(u, \lambda, \nu)$ is smooth and the symplectic gradient, or Hamiltonian vector field, $\overrightarrow{\mathscr{H}}^\nu(\cdot, u) : \T^*(\mathds{H}) \to \T^*(\mathds{H})$ is then the unique vector field on $\T^*(\mathds{H})$ satisfying
\[
\sigma_\lambda(\cdot, \overrightarrow{\mathscr{H}}^\nu(\lambda, u)) = \diff_{\lambda} \mathscr{H}^\nu(\cdot, u), \text{ for all } \lambda \in \T(\T^*(\H)),
\]
where $\sigma$ is the symplectic form of $\mathrm{T}^*(\mathbb{H})$.

We now state Pontryagin's maximum principle below (following \cite[Chapter 12]{Agrachev2004}), a first order necessary condition for solutions to the optimal control problem \cref{nonspacelikeOCP}.

\begin{theorem}[Pontryagin's Maximum Principle]
    \label{thm:PMP}
    Assume that the curve $\gamma : \interval{0}{T} \to \mathbb{H}$ controlled by $u \in \mathrm{L}^{\infty}(\interval{0}{T}, U)$ is optimal for the control problem given in \cref{nonspacelikeOCP}. Then, there exists a Lipschitz curve $\lambda : \interval{0}{T} \to \mathrm{T}^*(\mathbb{H})$, and a number $\nu \in \{-1, 0\}$ such that $\pi(\lambda(t)) = \gamma(t)$ for all $t \in \interval{0}{T}$ and
    \begin{theoenum}
    \item $\dot{\lambda}(t) = \overrightarrow{\mathscr{H}}^\nu(u(t), \lambda(t))$;
    \item $\mathscr{H}^\nu(u(t), \lambda(t)) = \max_{u \in \overline{U}} \mathscr{H}^\nu(u, \lambda(t)) = \mathrm{const}$;
    \item If $\nu = 0$, then $\lambda(t) \neq 0$ for all $t \in \interval{0}{T}$.
    \end{theoenum}
\end{theorem}

A Pontryagin extremal, or simply extremal, is a Lipschitz curve $\lambda : \interval{0}{T} \to \T^*(\H)$ such that there exists $u \in \mathrm{L}^{\infty}(\interval{0}{T}, U)$ for which (i), (ii) and (iii) of \cref{thm:PMP} are satisfied. An extremal is said to be a normal (resp. abnormal) extremal if it satisfies those conditions with $\nu = -1$ (resp. with $\nu = 0$). The characterisation of the extremals of the sub-Lorentzian Heisenberg group was carried out in \cite{sachkov2022sublorentzian}, and we repeat here some of details for the sake of clarity.

\begin{proposition}
	\label{prop:maxH}
    Given $\lambda \in \T^*(\H)$, if the maximum of $u \mapsto \mathscr{H}^\nu(u, \lambda)$ is attained for some $u \in U$ (resp. $u \in \mathrm{int}(U)$), then $\max_{u \in U} \mathscr{H}^\nu(u, \lambda) = 0$ (resp. $\max_{u \in \mathrm{int}(U)} \mathscr{H}^\nu(u, \lambda) = 0$) and $h_X(\lambda) \leq - \abs{h_Y(\lambda)}$ (resp. $h_X(\lambda) < - \abs{h_Y(\lambda)}$). More specifically, the maximum is attained at
    \begin{theoenum}
    \item $(u_1, u_2) = (0, 0)$ if $h_X(\lambda) < - \abs{h_Y(\lambda)}$;
    \item $(u, \mathrm{sgn}(h_Y(\lambda))u)$ for any $u \geq 0$ if $h_X(\lambda) = - \abs{h_Y(\lambda)}$;
    \item any $u \in \mathrm{int}(U)$ if $h_X(\lambda) = h_Y(\lambda) = 0$ and $\nu = 0$;
    \item $u = (u_1, u_2) \in \mathrm{int}(U)$ if $\nu = -1$ and if
    \begin{equation*}
    	h_X(\lambda) = -\frac{u_1}{\sqrt{u_1^2 - u_2^2}}, \text{ and } h_Y(\lambda) =\frac{u_2}{\sqrt{u_1^2 - u_2^2}}.
    \end{equation*}
    In this case, there are unique $A > 0$ and $\psi \in \R$ such that $h_X(\lambda) = - \cosh(\psi)$, $h_Y(\lambda) = \sinh(\psi)$, and $u = (A \cosh(\psi), A \sinh(\psi))$.
    \end{theoenum}
\end{proposition}

\begin{proof}
    If the maximum of $u \mapsto \max_{u \in U} \mathscr{H}^\nu(u, \lambda)$ is attained on the set $\{ (u_1, u_2) \mid u_1 = u_2 \geq 0 \}$, then in this case we have
    \begin{equation*}
    	\mathscr{H}^\nu(u, \lambda) = u_1 (h_X(\lambda) + h_Y(\lambda)),
    \end{equation*}
    and a maximum exists if and only if $h_X(\lambda) \leq - h_Y(\lambda)$. In this case, the maximum on this set is zero and is attained at $(u_1, u_2) = (0, 0)$ if $h_X(\lambda) < - h_Y(\lambda)$ and at any value $u_1 = u_2 \geq 0$ if $h_X(\lambda) = - h_Y(\lambda)$. Similarly, if the maximum is attained on the set $\{ (u_1, u_2) \mid u_1 = - u_2 \geq 0 \}$, then we have
    \begin{equation*}
    	\mathscr{H}^\nu(u, \lambda) = u_1 (h_X(\lambda) - h_Y(\lambda)),
    \end{equation*}
    and a maximum exists if and only if $h_X(\lambda) \leq h_Y(\lambda)$. In this case, the maximum on this set is zero and is attained at $(u_1, u_2) = (0, 0)$ if $h_X(\lambda) < h_Y(\lambda)$, and at any value $u_1 = - u_2 \geq 0$ if $h_X(\lambda) = h_Y(\lambda)$. In summary, the maximum on $\partial U$ exists if and only if $h_X(\lambda) \leq - \abs{h_Y(\lambda)}$, and when that is the case, the maximum is zero and is attained at $(u_1, u_2) = (0, 0)$ if $h_X(\lambda) < - \abs{h_Y(\lambda)}$, at $(u, u)$ for any $u \geq 0$ if $h_X(\lambda) = - h_Y(\lambda)$, and at $(u, -u)$ for any $u \geq 0$ if $h_X(\lambda) = h_Y(\lambda)$.

    Now, if a maximum is attained at $(u_1, u_2) \in \mathrm{int}(U)$, the first order condition implies that
    \begin{equation*}
    	h_X(\lambda) = \frac{\nu u_1}{\sqrt{u_1^2 - u_2^2}}, \text{ and } h_Y(\lambda) = -\frac{\nu u_2}{\sqrt{u_1^2 - u_2^2}}.
    \end{equation*}
    If $\nu = 0$, then we must have that $h_X(\lambda) = h_Y(\lambda) = 0$, and the maximum on $\mathrm{int}(U)$ is zero and attained for any $u \in \mathrm{int}(U)$. If $\nu = - 1$, then the maximum is also zero and writing $u_1 = A \cosh(\psi)$ and $u_2 = A \sinh(\psi)$ for the unique corresponding parameters $A > 0$ and $\psi \in \R$ yields
    \begin{equation*}
    	h_X(\lambda) = -\cosh(\psi), \text{ and } h_Y(\lambda) = \sinh(\psi),
    \end{equation*}
    which has a (unique) solution $\psi \in \R$ if and only if $h_X(\lambda)^2 - h_Y(\lambda)^2 = 1$ and $h_X(\lambda) \leq - 1$. Note that $h_X(\lambda) < - \abs{h_Y(\lambda)}$ also in this case.
\end{proof}

For a given Lipschitz curve $\lambda : \interval{0}{T} \to \T^*(\H)$, we will simply write $h_X(t)$, $h_Y(t)$ and $h_Z(t)$ instead of $h_X(\lambda(t))$, $h_Y(\lambda(t))$ and $h_Z(\lambda(t))$ respectively. The condition (i) of \cref{thm:PMP} can therefore be rewritten as
\begin{equation}
\label{eq:HamiltonequationsH}
    \begin{cases}
    \dot{h}_X = \{\mathscr{H}^\nu(\cdot, u), h_X\} = u_2 \{h_Y, h_X\} = -u_2 h_Z,\\
    \dot{h}_Y = \{\mathscr{H}^\nu(\cdot, u), h_Y\} = u_1 \{h_X, h_Y\} = u_1 h_Z,\\
    \dot{h}_Z = \{\mathscr{H}^\nu(\cdot, u), h_Z\} = 0, \\
    \dot{\gamma}(t) = u_1(t) X(\gamma(t)) + u_2(t) Y(\gamma(t)),
    \end{cases}
\end{equation}
where $\{\cdot,\cdot\}$ denotes the usual Poisson bracket on $C^\infty(\T^*(\H))$. Note that $h_Z(t) = h_Z(0)$ is constant by the third equation of \cref{eq:HamiltonequationsH}. We also introduce the ``maximised'' Hamiltonian $H : \T^*(\H) \to \R$ given by
\begin{equation}
\label{eq:HamiltonianH}
	H(\lambda) := \tfrac{1}{2}(h_Y^2(\lambda) - h_X^2(\lambda)).
\end{equation}
Note that this Hamiltonian is left-invariant.

We start by reviewing the abnormal case, following \cite[Theorem 1]{sachkov2022sublorentzian}.

\begin{proposition}
	\label{prop:abnormalextremals}
    Assume $\lambda : \interval{0}{T} \to \T^*(\H)$ is an abnormal extremal, associated with the control $u \in \L^\infty(\interval{0}{T}, U)$, such that the projection $\gamma(t) := \pi(\lambda(t))$ is never locally-constant (i.e. $u|_J \not\equiv 0$ on any sub-interval $J$ of $\interval{0}{T}$). Then there exists $T^* \in \interval{0}{T}$ such that
        \begin{theoenum}
    \item if $h_Z(t) = \mathrm{const} = 0$, then $(h_X(t), h_Y(t)) = \mathrm{const} \neq (0, 0)$, $h_X(t) = -\abs{h_Y(t)}$, and $u_1(t) = \mathrm{sgn}(h_Y(t)) u_2(t)$.
    \item if $h_Z(t) = \mathrm{const} > 0$, then
    \begin{equation*}
    	h_X(t) =
    	\begin{cases}
    		h_Y(t) < 0 & \text{for } t \in \ointerval{0}{T^*} \\
    		-h_Y(t) < 0 & \text{for } t \in \ointerval{T^*}{T}
    	\end{cases},
    	\text{ and }
    	u_1 =
    	\begin{cases}
    		- u_2 & \text{on } \ointerval{0}{T^*} \\
    		u_2 & \text{on } \ointerval{T^*}{T}
    	\end{cases}.
    \end{equation*}
    \item if $h_Z(t) = \mathrm{const} < 0$, then
    \begin{equation*}
    	h_X(t) =
    	\begin{cases}
    		- h_Y(t) < 0 & \text{for } t \in \ointerval{0}{T^*} \\
    		h_Y(t) < 0 & \text{for } t \in \ointerval{T^*}{T}
    	\end{cases},
    	\text{ and }
    	u_1 =
    	\begin{cases}
    		u_2 & \text{on } \ointerval{0}{T^*} \\
    		- u_2 & \text{on } \ointerval{T^*}{T}
    	\end{cases}.
    \end{equation*}
    \end{theoenum}
    In particular, $H(\lambda(t)) = 0$ and we have $u(t) \in \partial U$ for all $t \in \interval{0}{T}$.
\end{proposition}

\begin{proof}
    If $h_Z(t) = 0$, then Hamilton's equations \cref{eq:HamiltonequationsH} imply that both $h_X(t)$ and $h_Y(t)$ are constant. By the non-triviality property (iii) of \cref{thm:PMP}, $h_X(t)$ and $h_Y(y)$ cannot be both constantly zero. If $h_X(t) < -\abs{h_Y(t)}$, then \cref{prop:maxH} would imply that the extremal $\lambda(t)$ is trivial. Therefore, it must hold that $h_X(t) = -\abs{h_Y(t)}$ and $u_2(t) = \mathrm{sgn}(h_Y(t)) u_1(t)$.

	If $h_Z(t) > 0$, then \cref{eq:HamiltonequationsH} implies that $t \mapsto h_Y(t)$ is non-decreasing since $u_1(t) \geq 0$. If $h_Y(t)$ is constant on a sub-interval $J$ of $\interval{0}{T}$, then $u_1(t)$ is zero on $J$ and thus $u_2(t)$ is also zero on $J$, which would imply that $\gamma(t)$ is trivial on $J$. Therefore, $h_Y(t)$ is increasing in $t$ and may assume the value $h_Y(t) = 0$ at one $t \in \interval{0}{T}$ at most. This shows that the control $u(t)$ is not of the type (iii) in \cref{prop:maxH}. Assume now that $h_X(t_0) < - |h_Y(t_0)|$ for some $t_0 \in \interval{0}{T}$, then by continuity $h_X(t) < - |h_Y(t)|$ for all $t$ in a sub-interval of $\interval{0}{T}$ containing $t_0$. This would yield that $u_1(t) = u_2(t) = 0$ by (i) of \cref{prop:maxH} and so $\gamma(t)$ would be trivial on $J$. Therefore, $\lambda(t)$ must satisfy (ii) of \cref{prop:maxH}: $h_X(t) = - \abs{h_Y(t)}$ for all $t \in \interval{0}{T}$ and $u_1(t) = \mathrm{sgn}(h_Y(t)) u_2(t)$. The case $h_Z(t) < 0$ is treated similarly.
\end{proof}

We next characterise normal extremals, following \cite[Theorem 2]{sachkov2022sublorentzian}. By invariance under left translations, a curve $\gamma$ is a length-maximiser if and only if $q \cdot \gamma$ is a length-maximiser too. It is therefore enough to consider curves starting at the identity $e$ of the group $\H$.

\begin{proposition}
	\label{prop:normalextremals}
    Let $H > 0$ and $\lambda : \interval{0}{T} \to \T^*(\H)$ be a normal extremal associated with $u \in \L^\infty(\interval{0}{T}, \overline{U})$ such that the projection $\gamma(t) := \pi(\lambda(t))$ is never locally-constant and starts at $e =(0, 0, 0)$. Then, either $\lambda$ is also abnormal and of the form given in \cref{prop:abnormalextremals}, or else it can be reparametrised by constant speed $\sqrt{2H}$ and its reparametrisation satisfies $h_Z(t) = c$, $h_X(t) = - \cosh(\sqrt{2H} c t + \psi)$, $h_Y(t) = \sinh(\sqrt{2H} c t + \psi)$, and $\gamma(t) = (x(t), y(t), z(t))$ with
	\begin{equation}
	\label{eq:xyzcnotzero}
	\begin{dcases}
			x(t) = \dfrac{\sinh(\sqrt{2H} c t + \psi) - \sinh(\psi)}{c} \\
			y(t) = \dfrac{\cosh(\sqrt{2H} c t + \psi) - \cosh(\psi)}{c} \\
			z(t) = \dfrac{\sinh(\sqrt{2H} c t) - \sqrt{2H} c t}{2 c^2}
	\end{dcases}
	\end{equation}
	if $c \neq 0$, or $x(t) = \sqrt{2H} \cosh(\psi) t$, $y(t) = \sqrt{2H} \sinh(\psi) t$, and $z(t) = 0$ if $c = 0$, where $\psi \in \R$ is uniquely determined by $h_X(0) = - \cosh(\psi)$ and $h_Y(0) = \sinh(\psi)$. There is a unique control associated to this strictly normal extremal, which is given by
 \[
 u(t) = (-\sqrt{2H} h_X(t), \sqrt{2H} h_Y(t)) = (\sqrt{2H} \cosh(\sqrt{2 H} c t + \psi), \sqrt{2H} \sinh(\sqrt{2 H} c t + \psi)).
 \]
 In particular, along a strictly normal extremal, it holds $H(\lambda(t)) = \tfrac{1}{2}$ for all $t \in \interval{0}{T}$.
\end{proposition}

\begin{proof}
    The normal assumption means that $\nu = - 1$ in Pontryagin's Maximum Principle (\cref{thm:PMP}). If $h_X(t) = - \abs{h_Y(t)}$ for all $t \in \interval{0}{T}$, we obtain exactly the same (abnormal) extremals of \cref{prop:abnormalextremals}. Suppose therefore that there is a $t_0 \in \interval{0}{T}$ such that $h_X(t_0) < - \abs{h_Y(t_0)}$, and consider the maximal sub-interval $I$ of $\interval{0}{T}$ containing $t_0$ such that $h_X(t) < - \abs{h_Y(t)}$ for all $t \in I$ (which exists by continuity). For a given $t \in I$, we must then have (i) or (iv) of \cref{prop:maxH}. Since we are assuming that $\gamma(t)$ is never locally-constant, $u(t)$ can only vanish on a null subset of $I$, and thus by continuity of $\lambda(t)$, we deduce that (iv) of \cref{prop:maxH} must hold for all $t \in I$, and so $I = \interval{0}{T}$. \cref{prop:abnormalextremals} implies that this extremal is not abnormal because $H(\lambda(t)) = 1/2 \neq 0$. The reparametrisation $\varphi : \interval{0}{\L(\gamma)/\sqrt{2H}} \to \interval{0}{T}$ for which $\tilde{\gamma} := \gamma \circ \varphi :  \interval{0}{\L(\gamma)/\sqrt{2H}} \to \H$ has constant speed $\sqrt{2H}$ and the control $\tilde{u} : \interval{0}{\L(\gamma)/\sqrt{2H}} \to \R^2$ of $\tilde{\gamma}$ satisfies
	\[
        \dot{\varphi} = \frac{\sqrt{2H}}{\sqrt{(u_1 \circ \varphi)^2 - (u_2\circ \varphi)^2}}, \text{ and }
	\tilde{u} = \frac{\sqrt{2H}(u \circ \varphi)}{\sqrt{(u_1 \circ \varphi)^2 - (u_2\circ \varphi)^2}}.
	\]
	Thus $\tilde{u}_1 = -\sqrt{2H}\tilde{h}_X$ and $\tilde{u}_2 = \sqrt{2H} \tilde{h}_Y$ by (iv) of \cref{prop:maxH}.
	Hamilton's equation \cref{eq:HamiltonequationsH} becomes
	\begin{equation}
 \label{eq:HamiltonequationsHPMP}
	\begin{cases}
    \dot{\tilde h}_X = - \sqrt{2H} \ \tilde h_Y \tilde h_Z,\\
    \dot{\tilde h}_Y = -\sqrt{2H} \ \tilde h_X \tilde h_Z,\\
    \dot{\tilde h}_Z = 0, \\
    \dot{\tilde \gamma}(t) = \sqrt{2H} \left(-\ \tilde h_X(t) X(\tilde \gamma(t)) +\ \tilde h_Y(t) Y(\tilde \gamma(t)) \right).
    \end{cases}
	\end{equation}
Solving this system of differential equations finishes the proof of the claim.

\end{proof}

Based on the observations above, we introduce a sub-Lorentzian exponential map in the following way. For $t \in \R$ and $(q_0, \lambda_0) \in \T^*(\H)$, we denote by
\[
\lambda(t) := e^{t \overrightarrow{H}}(q_0, \lambda_0)
\]
the curve on $\T^*(\H)$ satisfying
\begin{equation}
\label{eq:hamiltoneqmaximised}
    \dot \lambda (t) = \overrightarrow{H}(\lambda(t)) \text{ with initial data }\lambda(0) = (q_0, \lambda_0).
\end{equation}
The projection of $\lambda(t)$ onto $\H$ will naturally be denoted by $\gamma(t) := \pi(\lambda(t))$. Written in the coordinates induced by the vector fields $X, Y$, and $Z$, the Hamiltonian equation \cref{eq:hamiltoneqmaximised} can be written as
\begin{equation}
\label{eq:HamiltonequationsH2}
    \begin{cases}
    \dot{h}_X = -h_Y h_Z,\\
    \dot{h}_Y = - h_X h_Z,\\
    \dot{h}_Z = 0, \\
    \dot{\gamma}(t) = - h_X(t) X(\gamma(t)) + h_Y(t) Y(\gamma(t)).
    \end{cases}
\end{equation}
Note that if $\lambda : \interval{0}{T} \to \T(\H)$ solves \cref{eq:HamiltonequationsH2}, then $H(\lambda(t))$ is constant for all $t \in \interval{0}{T}$ since
\begin{equation*}
	\frac{\diff}{\diff t} H(\lambda(t)) = h_Y(t) \dot h_Y(t) - h_X(t) \dot h_X(t) = 0.
\end{equation*}
Furthermore, the ``maximised'' Hamiltonian $H$ is homogeneous, in the sense that $H(q, \alpha \lambda_0) = \alpha^2 H(q, \lambda_0)$ for all $\alpha \in \R$ and all $(q, \lambda_0) \in \T^*(\H)$. In particular, if $\lambda(t)$ is a solution to \cref{eq:HamiltonequationsH2}, then $\lambda_{\alpha}(t) := \alpha \lambda(\alpha t)$ solves \cref{eq:HamiltonequationsH2} too, and thus
\[
e^{t \overrightarrow{H}}(q, \alpha \lambda_0) = \alpha e^{\alpha t \overrightarrow{H}}(q, \lambda_0).
\]
For the sake of completeness, we write down explicitly the solution to \cref{eq:HamiltonequationsH2}. Starting from the origin with initial covector $\lambda_0 = u_0 \diff x + v_0 \diff y + w_0 \diff z$, we get

\begin{equation}
\label{eq:xyzcnotzero2}
	\begin{dcases}
            h_Z(t) = w_0\\
            h_X(t) = u_0 \cosh(w_0 t) - v_0 \sinh(w_0 t)\\
            h_Y(t) = v_0 \cosh(w_0 t) - u_0 \sinh(w_0 t)
	\end{dcases}
 \ \text{, } \ \
 \begin{dcases}
			x(t) = \frac{v_0 (\cosh(w_0 t) - 1) - u_0 \sinh(w_0 t)}{w_0}\\
			y(t) = \frac{v_0 \sinh(w_0 t) - u_0 (\cosh(w_0 t) - 1)}{w_0} \\
			z(t) = (u_0^2 - v_0^2) \frac{\sinh(w_0 t) - w_0 t}{2 w_0^2}
	\end{dcases}.
\end{equation}

Since the flow of $H$ is an analytic function of the initial data, the above equations are understood by taking the limit $w_0 \to 0$ if $w_0 = 0$.

\begin{definition}
	The sub-Lorentzian exponental map of the Heisenberg group is defined by the projection of the flow of $H$ onto $\H$, i.e. given $t \in \R$
	\begin{equation}
		\exp^t_{q} : \T^*_{q}(\H) \to \H : \lambda_0 \mapsto \pi\left(e^{t\overrightarrow{H}}(q, \lambda_0)\right).
	\end{equation}
	As usual, $\exp_{q}$ will stand for $\exp^1_{q}$.
\end{definition}
The homogeneity of the Hamiltonian $H$ also implies that $\exp_{q}(t \lambda_0) = \exp^t_{q}(\lambda_0)$. By left-invariance, we also have that
\begin{equation}
	\label{eq:expqtoexpe}
	\exp_q(t \lambda_0) = q \cdot \exp_e(t (L_q)_e^*[\lambda_0]),
\end{equation}
for every $q \in \H$, $t \in \R$ and $\lambda_0 \in \T^*_q(\H)$. Curves $\gamma(t) = \exp_{q}(t \lambda_0)$ with $H(q, \lambda_0) > 0$ (resp. $< 0$) are timelike (resp. spacelike) curves, and those with $h_X(\lambda_0) < 0$ (resp. $> 0$) are future-directed (resp. past-directed). \cref{prop:normalextremals} shows that when $(q, \lambda_0) \in \T^*(\H)$, and $H := H(q, \lambda_0) > 0$, and $h_X(\lambda_0) < 0$, the curve
\[
t \mapsto \exp_{q}(t \lambda_0/\sqrt{2 H})
\]
is the projection of a strictly normal extremal parametrised by arclength and that
\[
t \mapsto \exp_q(t \lambda_0) = \exp_q^{\sqrt{2 H} t}(\lambda_0/\sqrt{2 H})
\]
is its reparametrisation by constant speed $\sqrt{2 H}$.

\begin{remark}
    Let $\gamma : \interval{0}{T} \to \mathbb{H}$ be a curve parametrised by constant speed $\sqrt{2 H}$ such that it is the projection of a strictly normal extremal $\lambda : \interval{0}{T} \to \T^*(\mathbb{H})$. Then, there is a $(q, \lambda_0) \in \T^*(\H)$ such that $\gamma(t) = \exp_q(t \lambda_0)$. Confusingly, we have
    \[
    \langle e^{t\overrightarrow{H}}(q_0, \lambda_0), V(\gamma(t)) \rangle = \sqrt{2 H} \ h_V(t), \text{ for } V = X, Y, \text{ and } Z.
    \]
    This is evident when comparing \cref{eq:HamiltonequationsHPMP} with \cref{eq:HamiltonequationsH2}. The Hamiltonian flow of $H$ includes both past- and future-directed curves, and it rescales the vertical component of the Pontryagin extremals from \cref{thm:PMP}. The authors of \cite{sachkov2022sublorentzian} define their exponential map by restricting its domain of definition to covectors that generate arclength curves, which avoids this rescaling. In contrast, we choose to define the sub-Lorentzian exponential map on the entire cotangent space, as this allows us to use the homogeneity of the Hamiltonian $H$ to simplify certain expressions.
\end{remark}

Let us summarise the results from \cite[Theorems 3, 7 and 8, Corollary 2]{sachkov2022sublorentzian} concerning the optimality of extremals.

\begin{theorem}
    \label{thm:geodesicspace}
    The sub-Lorentzian exponential map
    \[
    \exp_{q_0} : C_{q_0}(\H) := \left\{ \lambda_0 \in \T_{q_0}^*(\H) \mid h_X(\lambda_0) < - \abs{h_Y(\lambda_0)} \right\} \to I^+(q_0)
    \]
    is a real-analytic diffeomorphism. Furthermore, the future-directed timelike curves $\gamma(t) := \exp_{q_0}(t \lambda_0)$ are maximising geodesics for all $\lambda_0 \in C_{q_0}(\H)$ and all $t > 0$. For any $q \in J^+(q_0)$, there is a unique, up to reparametrisation, sub-Lorentzian length-maximiser that connects $q_0$ to $q$:
    \begin{theoenum}
    \item If $q \in I^+(q_0)$, then that trajectory is the future-directed timelike strictly normal trajectory given by $\gamma : \interval{0}{1} \to \H : t \mapsto \exp_{q_0}(t \lambda_0)$ where $\lambda_0 = \exp_{q_0}^{-1}(q)$.
    \item If $q_0 = (0, 0, 0)$ and $q = (x_1, y_1, z_1) \in J^+(q_0) \setminus I^+(q_0)$, then that trajectory is the future-directed strictly abnormal curve given by
    \begin{enumerate}[label=\arabic*), wide, labelindent=0pt]
        \item $\gamma : \interval{0}{x_1} \to \H : t \mapsto (t, \pm t, 0)$ if $z_1 = 0$;
        \item $\gamma(t) = (t, -t, 0) \mathbb{1}_{\interval{0}{\frac{x_1 - y_1}{2}}} + (t, t - (x_1 - y_1), \frac{x_1 - y_1}{2}(t - \frac{x_1 - y_1}{2}))\mathbb{1}_{\interval{\frac{x_1 - y_1}{2}}{x_1}}$ if $z_1 > 0;$
        \item $\gamma(t) = (t, t, 0) \mathbb{1}_{\interval{0}{\frac{x_1 + y_1}{2}}} + (t, x_1 + y_1 - t, -\frac{x_1 + y_1}{2}(t - \frac{x_1 + y_1}{2}))\mathbb{1}_{\interval{\frac{x_1 + y_1}{2}}{x_1}}$ if $z_1 < 0$.
    \end{enumerate}
    \end{theoenum}
\end{theorem}

\begin{remark}
    Note that \cref{thm:geodesicspace} shows that a maximising geodesic in the sub-Lorentzian Heisenberg group maintains the same causal character along the curve, i.e. it is either a timelike curve or a null curve. In metric spacetime geometry, having a causal character is usually a sign that the spacetime is regular enough; see \cite[Theorem 3.18]{kuzingersamann2018}.
\end{remark}

An explicit form of the causal and chronological future sets can easily be deduced from \cref{thm:geodesicspace}, also see \cite{Grochowski2006} for the details. By left-invariance, it is sufficient to write it down for $q_0 = (0, 0, 0) \in \H$:
\begin{equation}
    \label{eq:J+}
    J^{\pm}(q_0) := \left\{ (x, y, z) \in \H | -x^2 + y^2 + 4 |z| \leq 0, \pm x \geq 0 \right\}
\end{equation}
and
\begin{equation}
    \label{eq:I+}
    I^{\pm}(q_0) := \left\{ (x, y, z) \in \H | -x^2 + y^2 + 4 |z| < 0, \pm x > 0   \right\}.
\end{equation}

We end this part of this section by summarising important properties of the time-separation function $\tau$. The Lorentzian analogue of the following statement can be found in \cite[Theorem 3.6]{McCann2020}.

\begin{theorem}[Regularity of the sub-Lorentzian time-separation]
    \label{theorem:propertytau}
    The function $(q_0, q) \mapsto \tau(q_0, q)$ is continuous on $\H_{\leq}^2$ and real-analytic on $\H_{\ll}^2$. Therefore, it is also both locally semiconvex and locally semiconcave on $\H_{\ll}^2$.
\end{theorem}

The continuity and real-analyticity parts of the statement above is proven in \cite[Theorem 10]{sachkov2022sublorentzian}. Local semiconvexity and semiconcavity are implied from the smoothness of $\tau$ on $\H_{\ll}^2$, by \cite[Chapter 2]{Cannarsa2004}. For what concerns this work, it is the local semiconvexity of $q \mapsto \tau(q_0, q)$ on $I^+(q_0)$ that will be crucial. There is a useful formula for $\tau$ in \cite[Theorem 9]{sachkov2022sublorentzian}, which we reproduce here:
\begin{equation*}
	\label{eq:tauexplicit}
	\tau(e, (x, y, z)) = \sqrt{x^2 - y^2} \frac{\beta\left(\frac{z}{x^2 - y^2}\right)}{\sinh \beta \left(\frac{z}{x^2 - y^2}\right)}.
\end{equation*}
where $\beta : \ointerval{-1/4}{1/4} \to \R$ is the inverse function to the diffeomorphism
\[
\alpha(t) := \frac{\sinh(2 t) - 2 t}{8 \sinh^2(t)}.
\]
In particular, it holds that
\begin{equation}
\label{eq:tausamez}
	\tau((x_0, y_0, z), (x_1, y_1, z)) = \sqrt{(x_0 - x_1)^2 - (y_0 - y_1)^2},
\end{equation}
and that
\begin{equation}
\label{eq:tauleqptau}
	\tau((x_0, y_0, z_0), (x_1, y_1, z_1)) \leq \sqrt{(x_0 - x_1)^2 - (y_0 - y_1)^2},
\end{equation}
since $0 \leq x/\sinh(x) \leq 1$ for all $x \in \R$.

\subsection{The Lagrange multiplier rule}
\label{subsection:LagrangeMultiplier}

As in sub-Riemannian geometry, Pontryagin extremals can be characterised via Lagrange's multipliers rule, which will be useful later on. We start by defining the usual end-point map. 

\begin{definition}
The end-point map relative to the pair of times $(t_0, t_1) \in \interval{0}{T}^2$, and to the point $q \in \H$ is the map
\[
E^{t_0, t_1}_q : \L^\infty(\interval{0}{T}, \R^2) \to \H : u \mapsto \gamma_u(t_1)
\]
where $\gamma_u : \interval{0}{T} \to \H$ is the unique solution to the differential equation $\dot \gamma = u_1 X + u_2 Y$ with initial condition $\gamma(t_0) = q$.
\end{definition}

The end-point map is smooth and its differential is well-known.

\begin{proposition}[{\cite[Proposition 8.5]{comprehensive2020}}]
	\label{prop:diffEndpoint}
     The end-point map $E_{q}^{t_0, t_1}$ is smooth, for every $u \in \L^\infty(\interval{0}{T}, \R^2)$, its differential $D_u E_q^{t_0, t_1} : \L^\infty(\interval{0}{T}, \R^2) \to \T_{\gamma_u(t_1)}(\H)$ satisfies 
     \begin{equation}
     	\label{eq:diffEndpoint}
     	\Diff_u E_q^{t_0, t_1} (v) = \int_{t_0}^{t_1} (P_u^{t, t_1})_*\left[ v_1(t) X + v_2(t) Y \right] (\gamma_u(t_1)) \diff t.
     \end{equation}   
\end{proposition}

The sub-Lorentzian optimisation problem can reformulated as the problem of finding $u \in \L^{\infty}(\interval{0}{T}, \overline{U})$ that maximises the length $\L$ among all the controls $u \in (E^T_{q_0})^{-1}(q)$. The Lagrange multiplier rule, see \cite[Theorem 8.7]{comprehensive2020}, asserts that if $u \in \L^{\infty}(\interval{0}{T}, U)$ is such a maximiser, then there exists $\lambda_1 \in \T^*_q(\H)$ and $\nu \in \R$ such that $\lambda_1 \Diff_uE_{q_0}^T + \nu \Diff_u L = 0$. A small computation shows that when $u \in \L^{\infty}(\interval{0}{T}, U)$ describes a curve parametrised by arclength, then
\begin{align*}
	\Diff_u \L (v) = {} & \frac{\diff}{\diff s}  \L(u + s v)  \Big|_{s = 0} = \frac{\diff}{\diff s}  \int_{0}^T \sqrt{(u_1(t) + s v_1(t))^2 - (u_2(t) + s v_3(t))^2} \diff t  \Big|_{s = 0} \\
	={}& \int_{0}^T \frac{u_1(v) v_1(t) - u_2(t) v_2(t)}{\sqrt{u_1(t)^2 - u_2(t)^2}} \diff t = \int_0^T \left(u_1(t)v_1(t) - u_2(t) v_2(t) \right) \diff t.
\end{align*}
In light of this, we set
\[
(u, v)_{\mathrm{sL}} := \int_0^T (u_2(t)v_2(t) - u_1(t)v_2(t)) \diff t
\]
for all $u, v \in \mathrm{L}^{+\infty}(\interval{0}{T}, \R^2)$.

The Lagrange Multiplier rule is known to be related to Pontryagin's maximum principle. The following result is the analogue of \cite[Proposition 8.9]{comprehensive2020} for sub-Lorentzian strictly normal extremals.

\begin{proposition}
\label{prop:LagrangeMultiplier}
Let $\lambda : \interval{0}{T} \to \T^*(\H)$ be a Lipschitz curve and $\gamma : \interval{0}{T} \to \H$ be a horizontal curve controlled by $u \in \mathrm{L}^{\infty}(\interval{0}{T}, U)$, parametrised by arclength, and that joins $q_0$ to $q$ in $\H$. The following are equivalent.
	\begin{theoenum}
    \item The curve $\lambda$ is a strictly normal extremal associated with $u$.
    \item There exists $\lambda_T \in \T^*_{q}(\H)$ such that $\lambda(t) = (P_u^{t, T})^* (q, \lambda_T)$ for all $t \in \interval{0}{T}$ and
   	\begin{equation}
   		\label{eq:LagrangeMultiplier1}
   		\lambda_T \Diff_u E_{q_0}^{0, T} = (u, \cdot)_{\mathrm{sL}}.
   	\end{equation}
    \item There exists $\lambda_0 \in \T^*_{q_0}(\H)$ such that $\lambda(t) = (P_u^{t, 0})^* (q_0, \lambda_0)$ for all $t \in \interval{0}{T}$ and
    \begin{equation}
   		\label{eq:LagrangeMultiplier2}
   		\lambda_0 \Diff_u E_{q}^{T, 0} = -(u, \cdot)_{\mathrm{sL}}.
   	\end{equation}
    \end{theoenum}
\end{proposition}

\begin{proof}
	We start by showing that (ii) implies (i). Since $\lambda_T \in \T^*_{q}(\H)$ and $\lambda(t) = (P_u^{t, T})^* (q, \lambda_T)$ for all $t \in \interval{0}{T}$, we know by \cite[Proposition 4.12]{comprehensive2020} that the curve $\lambda$ satisfies Hamilton's equation \cref{eq:HamiltonequationsH} with initial value $\lambda(T) = (q, \lambda_T)$ and thus (i) of \cref{thm:PMP} holds. The identity \cref{eq:LagrangeMultiplier1} and \cref{prop:diffEndpoint} imply that for all $v \in \L^{\infty}(\interval{0}{T}, U)$, we have
	\begin{align*}
		\langle \lambda_T, D_u E^{0, T}_{q_0} (v) \rangle ={}& \int_0^T \langle \lambda_T, (P_u^{t, T})_*\left[ v_1(t) X + v_2(t) Y \right] (\gamma_u(T)) \rangle  \diff t\\
		={} & \int_0^T \langle \lambda(t), \left[ v_1(t) X(\gamma_u(t)) + v_2(t) Y(\gamma_u(t)) \right] \rangle  \diff t \\
		={} & \int_0^T v_1(t) h_X(t) + v_2(t) h_Y(t) \diff t = \int_0^T \left(u_2(t)v_2(t) - u_1(t) v_1(t) \right) \diff t.
	\end{align*}
	The last equality holds for all $v \in \L^{\infty}(\interval{0}{T}, \R^2)$ if and only if $u(t) = (- h_X(t), h_Y(t))$ for almost every $t \in \interval{0}{T}$. In particular, (ii) and (iii) of \cref{thm:PMP} are satisfied with $\nu = -1$, since (iv) of \cref{prop:maxH} is verified. By \cref{prop:normalextremals}, the extremal is strictly normal. The fact that (i) implies (ii) is obtained by performing the same computations in the reverse direction. The equivalence between (i) and (iii) is proven similarly, and the minus sign in \cref{eq:LagrangeMultiplier2} appears because, in \cref{eq:diffEndpoint}, the integral is taken from $T$ to $0$.
\end{proof}

In the next result, we fix $p \in \ointerval{0}{1}$, $q \in \H$ and $q' \in I^+(q)$, and we establish a relationship between the super-differentials of $x \mapsto \tau(x, q')^p/p$ at $q$ and the normal extremal whose projection is the maximiser (parametrised by arclength) joining $q$ to $q'$. This is the analogous to the geometric lemma found in \cite[Lemma 2.15]{rifford2014book}.

\begin{lemma}
    \label{lemma:Geometric}
    Let $q \in \H$, $q' \in I^+(q)$, and $\mathcal{O} \subseteq \H$ be an open neighbourhoof of $q$. Suppose that $\psi : \mathcal{O} \rightarrow \R$ is a map, differentiable at $q$, such that
    \[
    \psi(q) = \frac{1}{p} \tau(q,q')^p \text{ and } \psi(x) \geq \frac{1}{p} \tau(x ,q')^p \text{ for all } x \in \mathcal{O}.
    \]
    The unique maximising geodesic joining $q$ to $q'$ parametrised by arclength on $\interval{0}{T}$ is the projection of the strictly normal extremal $\lambda: [0,T] \to \mathrm{T}^*(\H)$ satisfying $\lambda(0)=(q,\Diff_{q} \psi/T^{p - 1})$. In particular, we have the identity
    \begin{equation}
    \label{eq:q'=T(q)}
    	q' = \exp_q\left(\Diff_{q} \psi \bigg/ \left(\sqrt{2 H(q, \Diff_q \psi)}\right)^{\tfrac{p - 2}{p-1}}\right).
    \end{equation} 
\end{lemma}

\begin{proof}
	Let $\gamma : \interval{0}{T} \to \H$ be the unique maximising geodesic joining $q$ to $q'$, which is strictly normal since $q' \in I^+(q)$. We will denote by $u_\gamma \in \L^{\infty}(\interval{0}{T}, U)$ the corresponding (unique) control parametrised by arclength. By construction, this control maximises the length functional $\L$ among all controls $u \in \L^{\infty}(\interval{0}{T}, U)$ satisfying $E^{T, 0}_{q'}(u) = q$. By definition of the time-separation function and then by hypothesis, it holds that
	\[
	\frac{1}{p}\L(u)^p \leq \frac{1}{p} \tau(E^{T, 0}_{q'}(u), q)^p \leq \phi(E^{T, 0}_{q'}(u))
	\]
 for every other control $u \in \L^\infty(\interval{0}{T}, U)$ with $E^{T, 0}_{q'}(u) \in \mathcal{O}$.
	Moreover, these inequalities become equalities at $u = u_\gamma$:
	\[
	\frac{1}{p}\L(u_\gamma)^p = \frac{1}{p} \tau(q, q')^p = \psi(q) = \psi(E^{T, 0}_{q'}(u_\gamma)).
	\]
	Therefore, $u_\gamma$ maximises the functional $u \mapsto \psi(E^{T, 0}_{q'}(u)) - \L(u)^p/p$ over the set of controls $u \in \L^\infty(\interval{0}{T}, U)$ satisfying $E^{T, 0}_{q'}(u) \in \mathcal{O}$, and thus $u_\gamma$ is a critical point of that functional. In particular, it holds that
	\[
	\Diff_{q} \psi \circ \Diff_{u_{\gamma}} E^{T, 0}_{q'} = - T^{p - 1} (u_\gamma, \cdot)_{\mathrm{sL}}, \text{ and consequently } \left(\frac{\Diff_{q} \psi}{T^{p - 1}} \right) \circ \Diff_{u_{\gamma}} E^{T, 0}_{q'} = -(u_\gamma, \cdot)_{\mathrm{sL}}.
	\]
	
	We set $\lambda_0 := \Diff_{q} \psi/T^{p - 1} \in \T^*_{q}(\H)$ and $\lambda(t) = (P_u^{0, t})^* (q, \lambda_0)$ for all $t \in \interval{0}{T}$, i.e. $\lambda : \interval{0}{T} \to \T^*(\H)$ is the curve satisfying (i) of \cref{thm:PMP} with $\lambda(0) = (q, \lambda_0)$. By \cref{prop:LagrangeMultiplier}, we deduce that $\lambda : \interval{0}{T} \to \T^*(\H)$ is the strictly normal extremal associated with $u$ whose projection is $\gamma$. In terms of the sub-Lorentzian exponential map, we need to follow the trajectory forward from $q$ for a time $T$ at constant speed 1 in order to reach $q'$, i.e.
    \[
    q' = \exp_q^{T}(\Diff_{q} \psi/T^{p - 1}) = \exp_q(\Diff_{q} \psi/T^{p - 2}).
    \]
    It remains to write $T$ in terms of $q$ and $\Diff_q \psi$. We showed in \cref{prop:normalextremals} that $H(\lambda(t)) = 1/2$ for all $t \in \interval{0}{T}$. So at $t = 0$, we have that $H(q, \lambda_0) = 1/2$ and thus
\[
\frac{1}{2} = \frac{1}{2}(h_Y^2(q, \lambda_0) - h_X^2(q, \lambda_0)) = \frac{1}{2 T^{2(p-1)}}(h_Y^2(q, \Diff_{q} \psi) - h_X^2(q, \Diff_{q} \psi)).
\]
Therefore, this implies that
\[
T^{2(p - 1)} = 2 H(q, \Diff_q \psi)
\]
and the proof is completed with an easy algebraic simplification.
\end{proof}

Naturally, a symmetric statement for the map $x \mapsto \tau(q, x)^p/p$ can be stated similarly by using \cref{prop:LagrangeMultiplier} (ii). We include it here without proof for the sake of completeness.

\begin{lemma}
    \label{lemma:Geometric2}
    Let $q \in \H$, $q' \in I^+(q)$, and $\mathcal{O} \subseteq \H$ be an open neighbourhoof of $q'$. Suppose that $\psi : \mathcal{O} \rightarrow \R$ is a map, differentiable at $q'$, such that
    \[
    \psi(q) = \frac{1}{p} \tau(q,q')^p \text{ and } \psi(x) \geq \frac{1}{p} \tau(q ,x)^p \text{ for all } x \in \mathcal{O}.
    \]
    The unique maximising geodesic joining $q$ to $q'$ parametrised by arclength on $\interval{0}{T}$ is the projection of the strictly normal extremal $\lambda: [0,T] \to \mathrm{T}^*(\H)$ satisfying $\lambda(T)=(q,-\Diff_{q} \psi/T^{p - 1})$. In particular, we have the identity
    \begin{equation*}
    	q = \exp_{q'}\left({\Diff_{q'} \psi} \bigg/ {\left(\sqrt{2H(q', \Diff_{q'} \psi)}\right)^{\tfrac{p - 2}{p-1}}}\right).
    \end{equation*} 
\end{lemma}

\subsection{The Lorentzian metric spacetime structure}
\label{subsection:HeisenbergMetricSpacetime}

In the last part of this section, we place the sub-Lorentzian Heisenberg group within the broader context of non-smooth Lorentzian geometry, following the general framework outlined in \cite{kuzingersamann2018, CavallettiMondino2022, CavallettiMondino2023}. {In order to have a suitable metric topology at hand, we will consider the sub-Riemannian distance $\diff$ on $\H$. The sub-Riemannian structure of the Heisenberg group is given by a Riemannian metric $\langle \cdot, \cdot \rangle_{\mathrm{sR}}$ on the subbundle $\mathcal{D} = \mathrm{span}\{X, Y\}$ of $\T(\H)$ such that the family $\{X, Y\}$ forms an orthonormal basis. The sub-Riemannian length of a horizontal curve $\gamma : \interval{0}{T} \mapsto \H$ controlled by $u = (u_1, u_2) \in \L^{\infty}(\interval{0}{T}, \R^2)$ is given by
    \[
    \L_{\mathrm{sR}}(\gamma) := \int_{0}^{T} \sqrt{u_1(t)^2 + u_2(t)^2} \diff t.
    \]
The sub-Riemannian distance is then given by 
\[
\diff(q_0, q) := \inf \left\{ \L_{\mathrm{sR}}(\gamma) \mid \gamma : \interval{0}{T} \to \H \text{ is horizontal, } \gamma(0) = q_0 \text{, and } \gamma(T) = q\right\}.
\]
It can be shown, by the Chow–Rashevskii theorem (see \cite[Theorem 3.31]{comprehensive2020}), that $\diff$ is indeed a distance function and induces a metric topology equivalent to the manifold topology. The metric space $(\H, \diff)$ is actually a length space, as noted in \cite[Section 3.7]{comprehensive2020}, and
\[
\L_{\mathrm{sR}}(\gamma) = \L_{\diff}(\gamma) := \sup \sum_{i = 0}^{N - 1}\diff(\gamma(t_i), \gamma(t_{i + 1})),
	\]
	where the infinimum is taken over all $N \in \mathds{N}$ and all finite partitions $(t_i)_{i = 0, \dots, N}$ of $\interval{0}{T}$ satisfying $t_0 = 0 < t_1 < \cdots < t_N = T$. We aim to keep the use of sub-Riemannian geometry to a minimum and refer interested readers to \cite{comprehensive2020} for further details.
}

\begin{definition}
	The two relations $\leq$ and $\ll$ on the sub-Lorentzian Heisenberg group $\H$ are defined as follows: for $q_0, q \in \H$, we set
    \begin{theoenum}
        \item $q_0 \leq q$ if and only if there exists a future-directed causal curve $\gamma: \interval{t_0}{T} \to \H$ with $\gamma(t_0)=q_0, \gamma(T)=q$.
        \item $q_0 \ll q$ if and only if there exists a future-directed timelike curve $\gamma: \interval{t_0}{T} \to \H$ with $\gamma(t_0)=q_0, \gamma(T)=q$.
    \end{theoenum}
    It is said that $q_0$ and $q$ are \textit{causally} (resp. \textit{timelike}) related if $q_0 \ll q$ (resp. $q_0 \leq q$).
\end{definition}

Given any subset $A \subseteq \H$, we also set
\[
A_{\leq}^2 := \{ (q_0, q) \in A^2 \mid q_0 \leq q \}, \ \ A_{\ll}^2 := \{ (q_0, q) \in A^2 \mid q_0 \ll q \}.
\]
It is clear that $\ll$ is transitive and contained in $\leq$, while $\leq$ is transitive and reflexive. This turns $(\H, \ll, \leq)$ into a \textit{causal space}, following \cite[Definition 1.1]{CavallettiMondino2022}. The \textit{causal future} and the \textit{chronological future} of $A$ as defined previously in \cref{section:subLgeometry} can be written as
\[
J^+(A) := \{ q \in \H \mid q_0 \leq q \text{ for some } q_0 \in A \}, \ \ I^+(A) := \{ q \in \H \mid q_0 \ll q \text{ for some } q_0 \in A \}.
\]
The same can be done for $J^-(A), I^-(A)$, and recall that $J^{\pm}(q_0) := J^{\pm}(\{q_0\})$ and that $I^{\pm}(q_0) := I^{\pm}(\{q_0\})$. It is clear that if $q_0 \leq q_1 \ll q_2$ or $q_0 \ll q_1 \leq q_2$, then $q_0 \ll q_2$, for all $q_0, q_1, q_2 \in \H$. 

The time-separation function introduced in \cref{timeseparation} enriches the structure of the causal space $(\H, \ll, \leq)$. It is not difficult to see from the work done in the previous sections that the map $\tau : \H \times \H \to \rinterval{0}{+\infty}$ is lower semicontinuous (actually continuous, see \cref{theorem:propertytau}), and that for every $q_0, q_1, q_2 \in \H$, it holds
\begin{theoenum}
        \item $\tau(q_0, q_1) = 0$ if $q_0 \nleq q$
        \item $\tau(q_0, q_1) > 0$ if and only if $q_0 \ll q_1$, and
        \item if $q_0 \leq q_1 \leq q_2$, then we have the \textit{reverse triangle inequality}
        \[
        \tau(q_0, q_2) \geq \tau(q_0, q_1) + \tau(q_1, q_2).
        \]
    \end{theoenum}
These properties of $\tau$ imply that $\ll$ is an open relation, i.e. the set $\{(q_0, q)\in \H^2 \mid \tau(q_0, q) > 0\}$ is open in $\H^2$, and that the sets $I^{+}(A)$ and $I^-(A)$ are open for every $A \subseteq \H$. By construction, $\tau(q, q) = 0$ and $\tau(q, q_0) = 0$ if $\tau(q_0 , q) > 0$, for all $q_0, q \in \H$. The structure $(\H, \diff, \ll, \leq, \tau)$ is then said to be a \textit{Lorentzian pre-length space}, in the sense of \cite[Definition 1.2]{CavallettiMondino2022}. 

We start by showing that the sub-Lorentzian Heisenberg group enjoys some important regularity properties.
\begin{proposition}
    \label{prop:HisGloballyHyperbolic}
	The sub-Lorentzian Heisenberg group is 
    \begin{theoenum}
        \item \emph{strongly causal}, i.e. the so-called Alexandrov topology obtained by taking $\{ I^+(q_0) \cap I^-(q) \mid q_0, q \in \H \}$ as a subbase coincides with the metric topology;
        \item \emph{causally closed}, i.e. $\H^2_\le$ is a closed set ($\leq$ is a closed relation in the product topology);
        \item \emph{compatible with the sub-Riemannian distance}, i.e. for every $q \in \H$, there exists a neighbourhood $\mathcal{U}$ of $q$ and a constant $C > 0$ such that $\L_{\diff}(\gamma) \leq C$ for all causal curves $\gamma$ with values in $\mathcal{U}$; 
        \item \emph{non-totally imprisoning}, i.e. for all compact subsets $K$ of $\H$, there is a uniform bound $C > 0$ for the sub-Riemannian lengths of causal curves $\gamma$ with values in $K$;
        \item \emph{globally hyperbolic}, i.e. it is non-totally imprisoning and also the sets $J^+(q_0) \cap J^-(q_1)$ are compact for all $q_0,q_1\in \H$;
        \item $\mathcal{K}$-\emph{globally hyperbolic}, i.e. it is non-totally imprisoning and also the sets $J^+(K_1) \cap J^-(K_2)$ are compact for all compact subsets $K_1, K_2 \subseteq \H$.
    \end{theoenum}
\end{proposition}

\begin{proof}

    \textbf{Strongly causal.} We have already noted that $I^+(q_0)$, $I^-(q)$, and thus $I^+(q_0) \cap I^-(q)$, are open for all $q_0, q \in \H$. Actually, the metric topology is always finer than the Alexandrov topology in a Lorentzian pre-length space by \cite[Section 2.4]{kuzingersamann2018}. It remains to show that for all $q_0 \in \H$, for all $\delta > 0$, there exists $A$ in the Alexandrov subbase such that $q_0 \in A$ and $A \subseteq B_{\diff}(q_0, \delta)$. By left translation and since the metric topology coincides with the manifold topology, it is enough to consider $q_0 = (0, 0, 0)$ and argue that for all $\delta > 0$, there is $q = (x_1, y_1, z_1) \in \H$ such that $I^+(q_0) \cap I^{-}(q)$ is contained in an Euclidean ball of radius $\delta$. We note that $I^-(q) = L_q(I^-(q_0))$ and a small computation then shows that
    \[
    I^-(q) = \left\{ (x, y, z) \in \H \mid -(x - x_1)^2 + (y - y_1)^2 + 4 \left|z - z_1 + \tfrac{1}{2}(x y_1 - x_1 y)\right| < 0, x < x_1  \right\}.
    \]
    Therefore, if $(x, y, z) \in I^+(q_0) \cap I^{-}(q)$, we must have that $0 < x < x_1$,
    $-x_1 < y < x_1$, and $-x_1^2/4 < z < x_1^2/2.$ In other words, we have the inclusion 
    \begin{equation}
        \label{eq:capI}
        I^+(q_0) \cap I^{-}(q) \subseteq \ointerval{0}{x_1} \times \ointerval{-x_1}{x_1} \times \ointerval{-x_1^2/4}{x_1^2/4}.
    \end{equation}
    
    \textbf{Causally closed.} Consider $(p_n)_{n \in \N}, (q_n)_{n \in \N} \subseteq \H$ with $p_n \leq q_n$ for all $n \in \N$ and such that $p_n \to p$ and $q_n \to q$ as $n \to \infty$. Then, we have that $q_0 := (0, 0, 0) = p_n^{-1} \cdot p_n \leq p_n^{-1} \cdot q_n$ for all $n \in \N$ by left-invariance. Since $J^+(q_0)$ is closed by \cref{eq:J+}, and since the left-translation is continuous, we obtain $q_0 \leq p^{-1} \cdot q$ and thus $p \leq q$.

    \textbf{Compatible with the sub-Riemannian distance.} Consider the Riemannian $\langle \cdot, \cdot \rangle_{\mathrm{R}}$ metric defined on the full tangent bundle $\T(\H) = \mathrm{span}\{X, Y, Z\}$ as the unique inner product that turns $X, Y,$ and $Z$ into an orthonormal basis. This Riemannian structure naturally extends the sub-Riemannian structure introduced at the beginning of this subsection, and of course $\L_{\mathrm{sR}}(\gamma) = \L_{\mathrm{R}}(\gamma)$ if $\gamma : \interval{0}{T} \to \H$ is horizontal. Denote by $g$ the Lorentzian metric on $\H$ satisfying $g(X, X) = -1$, $g(Y, Y) = 1$ and $g(Z, Z) = 1$. The Lorentzian metric $g$ naturally extends the sub-Lorentzian metric $\langle \cdot, \cdot \rangle$ as well. In particular, the Lorentzian and sub-Lorentzian lengths $\L_g$ and $\L$ coincide on horizontal curves too. Now, it is a consequence of \cite[Lemma 2.6.5]{piotr2011} that the Lorentzian Heisenberg group induced from $g$ is compatible with the Riemannian distance $\diff_{\mathrm{R}}$ obtained from $\langle \cdot, \cdot \rangle_{\mathrm{R}}$. This means that for all $q_0 \in \H$, there is a $\diff_{\mathrm{R}}$-open set $\mathcal{U}$ containing $q_0$ and a constant $C > 0$ such that every $g$-causal curve with values in $\mathcal{U}$ has $\diff_{\mathrm{R}}$-length bounded above by $C$. In particular, we also have $\L_{\diff}(\gamma) = \L_{\diff_{\mathrm{R}}}(\gamma) \leq C$ for every causal curve $\gamma : \interval{0}{T} \to \mathcal{U}$, since being causal with respect to the sub-Lorentzian metric implies being causal with respect to the Lorentzian metric $g$. The set $\mathcal{U}$ is also $\diff$-open since the sub-Riemannian and Riemannian topologies both coincide with the manifold topology.

    \textbf{Non-totally imprisoning.} A strongly causal, locally causally closed, and d-compatible Lorentzian pre-length space is non-totally imprisoning by \cite[Lemma 3.9]{felixgluing}.

    \textbf{Globally hyperbolic.} By left-translation, it is again enough to consider $q_0 = (0, 0, 0)$ and $q = (x, y, z) \in \H$. Recall also that the metric and manifold topologies coincide. From \cref{eq:capI}, we know that $J^+(q_0) \cap J^-(q)$ is also bounded and since it is the intersection of two closed sets by \cref{eq:J+}, it is compact too.
    
    \textbf{$\mathcal{K}$-globally hyperbolic.} A globally hyperbolic and causally closed Lorentzian pre-length space satisfying $I^{\pm}(x) \neq \emptyset$ (which is true in our case by \cref{eq:I+}) is $\mathcal{K}$-globally hyperbolic by \cite[Proposition 1.6]{CavallettiMondino2022} (see also \cite[Remark 2.5]{braun2022good}).
\end{proof}

The following proposition clarifies the relationship between the sub-Lorentzian geometry built upon the optimal control problem \cref{nonspacelikeOCP} and the non-smooth setting of $(\H, \diff, \ll, \leq, \tau)$ as a Lorentzian pre-length space.

\begin{proposition}
\label{prop:lengthproperty}
	A curve $\gamma : \interval{0}{T} \to \H$ is a future-directed causal curve (resp. timelike curve) if and only if $\gamma$ is locally Lipschitz with respect to the metric topology induced by $\diff$ and $\gamma(s) \leq \gamma(t)$ (resp. $\gamma(s) \ll \gamma(t)$) for all $s, t \in \interval{0}{T}$ such that $s < t$. Furthermore, the length of a future-directed causal curve $\gamma : \interval{0}{T} \to \H$ as in \cref{eq:lorentzianlength} can be obtained as
    \begin{equation}
    \label{eq:L=Ltau}
        \mathrm{L}(\gamma) = \L_{\tau}(\gamma) := \inf \sum_{i = 0}^{N - 1}\tau(\gamma(t_i), \gamma(t_{i + 1})),
    \end{equation}
	where the infinimum is taken over all $N \in \mathds{N}$ and all finite partitions $(t_i)_{i = 0, \dots, N}$ of $\interval{0}{T}$ satisfying $t_0 = 0 < t_1 < \cdots < t_N = T$.
\end{proposition}

\begin{proof}
    If $\gamma$ is future-directed and causal (resp. timelike), then given $s, t \in \interval{0}{T}$ with $s < t$, the curve $\gamma|_{\interval{s}{t}}$ is of course a future-directed and causal (resp. timelike) curve joining $\gamma(s)$ to $\gamma(t)$, and so $\gamma(s) \leq \gamma(t)$ (resp. $\gamma(s) \ll \gamma(t)$). 

    Suppose that $\gamma : \interval{0}{T} \to \H$ is locally Lipschitz, and such that $\gamma(s) \leq \gamma(t)$ (resp. $\gamma(s) \ll \gamma(t)$) for all $s, t \in \interval{0}{T}$ with $s < t$. In particular, the curve $\gamma$ is Lipschitz in charts and, by Rademacher's theorem, it is absolutely continuous and hence differentiable almost everywhere. Actually, being locally Lipschitz with respect to the sub-Riemannian distance implies that $\gamma$ is globally Lipschitz with respect to $\diff$,
    and \cite[Proposition 3.50]{comprehensive2020} yields that $\gamma$ is a horizontal curve, i.e. there exists $u \in \L^{\infty}(\interval{0}{T}, \R^2)$ such that $\dot{\gamma}(t) = u_1(t) X(\gamma(t)) + u_2(t) Y(\gamma(t))$ for almost every $t \in \interval{0}{T}$. It remains to show that $u(t) \in \overline{U}$ (resp. $u(t) \in U$) for almost every $t \in \interval{0}{T}$. 
    
We follow some of the ideas from the proofs of \cite[Theorem 3.41 and Proposition 3.50]{comprehensive2020}. Consider the partition $\sigma_n := (t_i^n)_{i = 1, \dots, n}$ of the interval $\interval{0}{T}$ into $2^n$ intervals of length $T/2^n$, namely $t_i^n := i T/2^n$ for $i = 1, \dots, 2^n$. Denote by $\gamma_n : \interval{0}{T} \to \H$ the curve defined by the concatenation of length-maximisers parametrised on $\interval{t_i^n}{t_{i + 1}^n}$ by constant speed joining $\gamma(t^n_i)$ to $\gamma(t_{i + 1}^n)$ for $i = 1, \dots, 2^n - 1$. The sequence of curves $(\gamma_n)_{n \in \N}$ converges to $\gamma$ pointwise by construction. Furthermore, the curves $\gamma_n : \interval{0}{T} \to \H$ lie in the set $J^+(q_0) \cap J^-(q)$, which is compact since we have shown in (v) of \cref{prop:HisGloballyHyperbolic} that the sub-Lorentzian Heisenberg group is globally hyperbolic. It is also non-totally imprisoning by (iv) of the same result, and we thus have that $\L_{\diff}(\gamma_n) \leq C$ for some $C > 0$ and all $n \in \N$, i.e. the curves $\gamma_n$ have uniformly bounded lengths. Since the sub-Riemannian Heisenberg group $(\H, \diff)$ is a complete locally compact length space, the Arzela-Ascoli
theorem from \cite[Theorem 2.5.14]{burago-burago-sergei2001} implies that the sequence of curves $\gamma_n$ contains a uniformly convergent subsequence, which must coincide with $\gamma$. With a slight abuse of notation, we continue to denote this subsequence by $(\gamma_n)_{n \in \N}$.
For almost every $t \in \interval{0}{T}$, we have
    \[
    \dot{\gamma}_n(t) = u_1^n(t) X(\gamma_n(t)) + u_2^n(t) Y(\gamma(t)),
    \]
    where for all $n \in \mathbb{N}$ and $i = 1, \dots, 2^n - 1$
    \[
    u_1^n(t) \geq |u_2^n(t)| \text{ (resp. $u_1^n(t) > |u_2^n(t)|$)}.
    \]
    Thus, we have that $\dot{\gamma}_n(t) \in V_{\gamma_n(t)}$ for almost every $t \in \interval{0}{T}$, where
    \begin{equation}
        \label{def:Vq}
        V_q := \left\{ f_u(q) := u_1 X(q) + u_2 Y(q) \ | \ u_1 \geq |u_2| \text{ (resp. $u_1 > |u_2|$)} \right\}.
    \end{equation}
    Note that $V_q$ is a convex set for every $q \in \H$. For $t_0 \in \interval{0}{T}$ and all $h > 0$ small enough, we may write in local coordinates that
    \begin{equation}
        \label{eq:tobepassedtolimit}
        \frac{1}{h}(\gamma_n(t_0 + h) - \gamma_n(t_0)) = \int_{t_0}^{t_0 + h} f_{u_n(t)}(\gamma_n(t)) \diff t \in \mathrm{Conv} \left( V_{\gamma_n(t)} \ | \ t \in \interval{t_0}{t_0 + h} \right).
    \end{equation}
    Continuing in local coordinates, by uniform convergence, we have that for $n \geq n(h)$ , the following holds:
    \[
    |\gamma_n(t) - \gamma(t)| \leq h, \ \text{ for all } t \in \interval{t_0}{t_0 + h}.
    \]
    Furthermore, since $\gamma$ is Lipschitz, there exists a constant $L > 0$ such that
    \[
    |\gamma(t) - \gamma(t_0)| \leq L h, \ \text{ for all } t \in \interval{t_0}{t_0 + h}.
    \]
    In particular, we have that
    \[
    |\gamma_n(t) - \gamma(t_0)| \leq |\gamma_n(t) - \gamma(t)| + |\gamma(t) - \gamma(t_0)| \leq (L + 1)h,
    \]
    and thus $\gamma_{n(h)}(t_0 + h) \to \gamma(t_0)$ as $h \to 0$. Since the metric topology and the manifold topology coincide, we deduce that for all $t \in \interval{t_0}{t_0 + h}$ and for all $n \geq n(h)$, it holds that $\gamma_n(t) \in B_{\diff}(\gamma(t_0), r(h))$, where $r(h) \to 0$ as $h \to 0$. We obtain that
\[
\mathrm{Conv} \left( V_{\gamma_n(t)} \ | \ t \in \interval{t_0}{t_0 + h} \right) \subseteq \mathrm{Conv} \left( V_{q} \ | \ q \in B_{\diff}(\gamma(t), r(h)) \right),
\]
and when $t_0$ is a differentiable point of $\gamma$, we may pass \cref{eq:tobepassedtolimit} to the limit as $h \to 0$ to find that
\[
\dot{\gamma}(t_0) \in \mathrm{Conv} \left( V_{\gamma(t_0)} \right) = V_{\gamma(t_0)}.
\]
Recalling \cref{def:Vq} shows that $\gamma$ is a future-directed causal curve.

    Lastly, we prove the identity \cref{eq:L=Ltau}. Let $\gamma : \interval{0}{T} \to \H$ be a future-directed causal curve and denote by $g$ the Lorentzian metric on $\H$ satisfying $g(X, X) = -1$, $g(Y, Y) = 1$ and $g(Z, Z) = 1$. As seen before, the Lorentzian metric $g$ naturally extends the sub-Lorentzian metric $\langle \cdot, \cdot \rangle$. In particular, the Lorentzian and sub-Lorentzian lengths $\L_g$ and $\L$ coincide on horizontal curves. Furthermore, the time-separation function $\tau_g$ induced from this Lorentzian metric is obtained by replacing in \cref{eq:tautimeseparation} $\L$ by $\L_g$ and $\Omega_{q_0 q}$ by the set $\Omega_{q_0 q}^g$ of future-directed absolutely continuous curves $\gamma : \interval{0}{T} \to \H$ satisfying $g(\dot{\gamma}(t), \dot{\gamma}(t)) \leq 0$ and $\dot{\gamma}(t) \neq 0$ for almost every $t \in \interval{0}{T}$ (that is to say, the future-directed $g$-causal curves). Since every horizontal causal curve is also $g$-causal, we must have $\tau \leq \tau_g$. On the one hand, we know by \cite[Proposition 2.32]{kuzingersamann2018} that
    \[
	\mathrm{L}(\gamma) = \mathrm{L}_g(\gamma) = \inf \sum_{i = 0}^{N - 1}\tau_g(\gamma(t_i), \gamma(t_{i + 1})) \geq \inf \sum_{i = 0}^{N - 1}\tau(\gamma(t_i), \gamma(t_{i + 1})).
	\]
    On the other hand, we also have that
    \[
    \inf \sum_{i = 0}^{N - 1}\tau(\gamma(t_i), \gamma(t_{i + 1})) \geq \inf \sum_{i = 0}^{N - 1}\L(\gamma|_{\interval{t_i}{t_{i + 1}}}) = \L(\gamma).
    \]
\end{proof}

In other words, \cref{prop:lengthproperty,prop:HisGloballyHyperbolic} are saying that the Lorentzian pre-length space $(\H, \diff, \ll, \leq, \tau)$ is actually a \textit{Lorentzian length space}, see \cite{kuzingersamann2018, CavallettiMondino2022, CavallettiMondino2023}. The notion of maximising geodesic can also be reinterpreted in this context, and the following statement is a direct consequence of \cref{prop:lengthproperty}.

\begin{proposition}
	A curve $\gamma : \interval{0}{T} \to \H$ is a timelike maximising geodesic  if and only if for all $s, t \in \interval{0}{T}$ with $s < t$, it holds
	\[
	\tau(\gamma(s), \gamma(t)) > 0 \text{, and } \tau(\gamma(0), \gamma(T)) = \mathrm{L}(\gamma).
	\] 
	Furthermore, a curve $\gamma : \interval{0}{T} \to \H$ is a timelike maximising geodesic parametrised by arclength if and only if for all $s, t \in \interval{0}{T}$ with $s < t$, it holds 
	\[
	\tau(\gamma(s), \gamma(t)) > 0 \text{, and } \tau(\gamma(s), \gamma(t)) = t - s.
	\]
\end{proposition}

\cref{thm:geodesicspace} says that for every $q_0, q \in \H$ such that $q_0 \leq q$ there exists a maximising timelike geodesic parametrised by arclength joining $q_0$ to $q$. In the language of non-smooth Lorentzian geometry, this means that the Lorentzian length space $(\H, \diff, \ll, \leq, \tau)$ is a \textit{Lorentzian geodesic space} (again see \cite{kuzingersamann2018, CavallettiMondino2022, CavallettiMondino2023}).

 \section{Optimal Transport in the sub-Lorentzian Heisenberg group}
 \label{section:OTinSLH}

\subsection{Lorentzian optimal transport in the Heisenberg group}

In this section, we introduce the Lorentzian optimal transport problem in the Heisenberg group, together with details on Kantorovich duality. The results presented here are formulated for the sub-Lorentzian Heisenberg group, although they also hold in the more general setting of Lorentzian metric spaces. Further details and proofs can be found in \cite{McCann2020, CavallettiMondino2022, CavallettiMondino2023}. We will assume that $p \in \linterval{0}{1}$, unless stated otherwise.

Given two Borel probability measures $\mu, \nu \in \mathcal{P}(\H)$, the (forward) Lorentz-Monge problem consists of maximising the functional
\begin{equation}
	\label{eq:LorentzMonge}
	M^+(T) := \int_{\H} c_p(x, T(x)) \diff \mu(x)
	\tag{M$^+$}
\end{equation}
among all the \emph{(forward) transport maps} from $\mu$ to $\nu$, i.e. the measurable map $T: \H \to \H$ such that $T_\sharp \mu = \nu$ and $T(q) \in J^+(q)$ for $\mu$-almost every $q \in \H$. A map that realises the maximum of \cref{eq:LorentzMonge} is called an \emph{(forward) optimal transport map} from $\mu$ to $\nu$. Similarly, the backward Lorentz-Monge problem consists of maximising the functional
\begin{equation}
	\label{eq:LorentzMonge-}
	M^-(T) := \int_{\H} c_p(T(x), x) \diff \nu(x)
	\tag{M$^-$}
\end{equation}
among all the \emph{(backward) transport maps} from $\nu$ to $\mu$, i.e. the measurable map $T: \H \to \H$ such that $T_\sharp \nu = \mu$ and $T(q) \in J^-(q)$ for $\nu$-almost every $q \in \H$.

This optimisation problem, whose Euclidean version was originally introduced by the French Mathematician Gaspard Monge in 1781, is not easy to solve directly with the usual methods of Calculus of Variations because the constraints are not closed with respect to the weak topology. Instead, it is easier to first analyse a relaxation, which was initially studied by Kantorovich in 1942. The idea is to define the set of \emph{transport plans} between $\mu$ and $\nu$ as
\[
    \Pi(\mu, \nu) = \left\{ \pi \in \mathcal{P}(\H \times \H) \mid (P_1)_\# \pi = \mu, (P_2)_\# \pi = \nu \right\},
    \]
 where $P_1,P_2:\H \times \H \to \H$ are the projections onto the respective coordinate, i.e. $P_1(q_0, q_1) := q_0$ and $P_2(q_0, q_1) := q_1$, for all $q_0, q_1 \in \H$.
 We also introduce the set of causal  (respectively timelike) transport plans from $\mu$ to $\nu$ as
    \[
    \text{ and } \ \Pi_\le(\mu, \nu) = \left\{ \pi \in \Pi(\mu,\nu) \mid \pi(\H^2_{\le}) = 1 \right\}, \ \ \Pi_\ll(\mu, \nu) = \left\{ \pi \in \Pi(\mu,\nu) \mid \pi(\H^2_{\ll}) = 1 \right\}.
    \]
 The Lorentz-Kantorovich problem then consists of maximising the functional
 \begin{equation}
 	\label{eq:LorentzKantorovich}
 	K(\pi) := \int_{\H \times \H} c_p(x, y) \diff \pi(x, y) \tag{K}
 \end{equation}
 among all $\pi \in \Pi_\le(\mu,\nu)$. A plan $\pi$ realising the maximum in \cref{eq:LorentzKantorovich} is called an \emph{optimal transport plan} from $\mu$ to $\nu$. Note that if $T$ is an optimal transport map for \cref{eq:LorentzMonge} (resp. for \cref{eq:LorentzMonge-}), then $\pi := (\mathrm{Id} \times T)_\sharp \mu$ (resp. $\pi := (T \times \mathrm{Id})_\sharp \nu$) is optimal for \cref{eq:LorentzKantorovich}.
 
 The advantage of this relaxation is that the constraint $\Pi_\le(\mu,\nu)$ is compact for the weak topology, and thus a proof for the existence of a maximiser follows from a standard argument. The next statement follows immediately from \cite[Proposition 1.5]{CavallettiMondino2022}, since we have shown in \cref{prop:HisGloballyHyperbolic} that $\H$ is a globally hyperbolic Lorentzian geodesic space.

\begin{proposition}
    \label{prop:optimalexistence}
    Let $\mu,\nu\in\mathcal{P}(\H)$ and suppose that 
    \begin{theoenum}
    	\item $\Pi_\le(\mu,\nu)\neq \varnothing$,
    	\item there exist measurable functions $a,b:\H \to \R$ with $a\oplus b \in \L^1(\mu\otimes\nu)$ and $\tau(q, q')^p \leq a(q) + b(q')$ for all $(q, q') \in \mathrm{supp}(\mu) \times \mathrm{supp}(\nu)$, where $a \oplus b$ denotes the function defined by $a \oplus b(q, q') = a(q) + b(q')$.
    \end{theoenum}
    Then the supremum in \cref{eq:LorentzKantorovich} is attained and finite.
\end{proposition}
A general assumption  that is reasonable to assume on $\mu$ and $\nu$, other than just assuming $\Pi_{\leq}(\mu, \nu) \neq \varnothing$, is that
\begin{equation}
\label{remark:condition(i)}
	\mathrm{supp}(\mu) \times \mathrm{supp}(\nu) \subseteq \H_{\leq}^2.
	\tag{I}
\end{equation}
Indeed, if that is the case, then $\mu \otimes \nu \in \H_{\leq}^2$. Note that by the reverse triangle inequality of $\tau$, we have
	\[
	\tau(q, q')^p \leq (\tau(q, q'') - \tau(q', q''))^p \leq 2^p (\tau(q, q'')^p + \tau(q', q'')^p)
	\]
	and
	\[
	\tau(q, q')^p \leq (\tau(q'', q') - \tau(q'', q))^p \leq 2^p (\tau(q'', q')^p + \tau(q'', q)^p).
	\]
Thus if $\mu, \nu \in \mathcal{P}(\H)$ are such that there exists $q'' \in \H$ with
\begin{equation}
\label{remark:condition(ii)}
	\begin{split}
		\min\Big(\int_{\H} \tau(q, q'')^p \diff \mu(q) +& \int_{\H} \tau(q', q'')^p \diff \nu(q'), \\
		&\int_{\H} \tau(q'', q)^p \diff \mu(q) + \int_{\H} \tau(q'', q')^p \diff \nu(q')\Big) < +\infty,
	\end{split}
	\tag{II}
\end{equation}
	then (ii) of \cref{prop:optimalexistence} would be satisfied. Since $\tau$ is continuous, this is the case if both $\mu$ and $\nu$ have compact support.

The relaxation \cref{eq:LorentzKantorovich} also defines a Lorentzian metric on the space of probability measures $\mathcal{P}(\H)$.

\begin{definition}
    \label{plorentzwasserstein}
    Given $\mu,\nu \in \mathcal{P}(\H)$, the $p$-Lorentzian cost is given by $c_p(x, y) = \tau(x, y)^p/p$. The $p$-Lorentz-Wasserstein distance on $\H$, associated with the cost $c_p$, is then defined as
    \begin{equation}
    	\label{eq:LWDistTau}
    	\ell_p(\mu, \nu) = \sup_{\pi \in \Pi_\le(\mu,\nu)} \left( \int_{\H \times \H} c_p(x,y) \pi(\diff x \diff y) \right)^{1/p} \quad \text{ if } \Pi_\le(\mu,\nu) \neq \emptyset.
    \end{equation}
    By convention, we set $\ell_p(\mu,\nu) = -\infty$ if $\Pi_\le(\mu,\nu) = \emptyset$.
\end{definition}

Analogously to the well-known Kantorovich-Rubinstein-Wasserstein distances in classical optimal transport theory, the Lorentz-Wasserstein distance $\ell_p$ satisfies the reverse triangle inequality; see \cite[Proposition 2.5]{CavallettiMondino2023}. Specifically, for all $\mu_0, \mu_1, \mu_2 \in \mathcal{P}(\H)$, we have
\[
\ell_p(\mu_0, \mu_1) + \ell_p(\mu_1, \mu_2) \leq \ell_p(\mu_0, \mu_2),
\]
where by convention we set $\infty - \infty = - \infty$.
The space $\mathcal{P}(\H)$ of probability measures on $\H$ can thus be endowed with a Lorentzian structure induced from $\ell_p$, and we call this structure the Lorentz-Wasserstein space. Details on the geometry of the Lorentz-Wasserstein space can be found in \cite{CavallettiMondino2022, CavallettiMondino2023}. We say that $(\mu_t)_{t \in \interval{0}{1}} \subseteq \mathcal{P}(\H)$ is an $\ell_p$-geodesic if
\begin{equation}
	\label{eq:Wassersteingeodesic}
	\ell_p(\mu_s, \mu_t) = (t - s) \ell_p(\mu_0, \mu_1), \qquad \text{ for all } t, s \in \interval{0}{1} \text{ with } s \leq t.
\end{equation}

The standard concept of cyclical monotonicity can also be introduced in this context, following \cite[Section 2.2]{CavallettiMondino2023}.
\begin{definition}
	\label{def:cyclicallymonotone}
	A set $\Gamma \subset \H^2_\le$ is $c_p$-cyclically monotone
    if we have
    \[
        \sum_{i=1}^N c_p(x_i, y_i) \geq \sum_{i=1}^N c_p(x_{i+1},y_i),
    \]
	for all $N\in\mathbb{N}$ and all $(x_1, y_1),\ldots,(x_N, y_N)\in \Gamma$
    such that $x_{N+1}=x_1$. A coupling $\pi \in \mathcal{P}(\H \times \H)$ is said to be $c_p$-cyclically monotone if it is concentrated on a $c_p$-cyclically monotone set, i.e. if there is a $c_p$-cyclically monotone set $\Gamma \subseteq \H_{\leq}^2$ such that $\pi(\Gamma) = 1$.
\end{definition}

Note that since $c_p$ is continuous by \cref{theorem:propertytau}, any $c_p$-cyclically monotone set is closed. In particular, a coupling $\pi$ is $c_p$-cyclically monotone if and only if $\mathrm{supp}(\pi)$ is $c_p$-cyclically monotone. Kantorovich found out that cyclical monotonicity was actually closely related to optimality. The same can be said in the Lorentzian setting, see \cite[Proposition 2.8]{CavallettiMondino2023}.

\begin{proposition}
	\label{prop:CMandOptimality}
	Let $\mu,\nu\in\mathcal{P}(\H)$. Suppose that $\Pi_\le(\mu,\nu)\neq \varnothing$ and that there exist measurable functions $a,b:\H \to \R$, with $a\oplus b \in \L^1(\mu\otimes\nu)$ and $\tau^p(q_0, q) \leq a(q_0) + b(q)$ for all $(q_0, q) \in \mathrm{supp}(\mu) \times \mathrm{supp}(\nu)$. Then the following holds.
\begin{theoenum}
        \item If $\pi$ is an optimal transport plan from $\mu$ to $\nu$ and $P_1(\Gamma) \times P_2(\Gamma) \subseteq \H_\leq^2$, then it is $c_p$-cyclically monotone.
        \item If $\pi$ is $c_p$-cyclically monotone and $\pi(\H^2_{\ll})=1$, then $\pi$ is an optimal transport plan from $\mu$ to $\nu$.
\end{theoenum}
\end{proposition}

Our next goal is to introduce the concept of duality, a fundamental result in optimal transport theory.  The dual problem consists in minimising the functional
 \begin{equation}
 	\label{eq:dualproblem}
 	D(u, v) := \int_{\H} u \diff \mu + \int_{\H} v \diff \nu
 	\tag{D}
 \end{equation}
 among the pair of functions $u, v : \H \to \R \cup \left\{+ \infty \right\}$ with $u \oplus v \geq c_p$ on $\mathrm{supp}(\mu) \times \mathrm{supp}(\nu)$ and $u \oplus v \in \mathrm{L}^1(\mu \otimes \nu)$. When $\sup K(\pi) = \inf D(u, v)$, we say that Kantorovich duality holds for the pair of measures $(\mu, \nu)$. When $\max K(\pi) = \min D(u, v)$, we say that strong Kantorovich duality holds for the pair of measures $(\mu, \nu)$.

It is therefore natural to ask whether the infimum in \cref{eq:dualproblem} is attained for some functions $u$ and $v$. The following definition is key to addressing this question.

\begin{definition}[$c_p$-transform, $c_p$-concavity, $c_p$-subdifferential]
\label{def:subdifferential}
    Let $A_1$, $A_2$ be non-empty subsets of $\H$, and $\varphi: A_1 \to \R$ a function.
    The function
    \[
    \varphi^{(c_p)}: A_2 \to \R \cup \{+\infty\} : y \mapsto \sup \left\{ \varphi(x) + c_p(x, y) \mid x \in A_1 \right\}
    \]
    is the \textit{$c_p$-transform of $\varphi$ (relative to $(A_1,A_2)$)}. The function $\varphi$ is \textit{$c_p$-concave (relative to $(A_1,A_2)$}) if there is a function $\psi: A_2 \to \R \cup \{+\infty\}$ with
    \begin{equation}
        \label{eq:taupconcave}
        \varphi(x) = \inf \left\{ \psi(y) - c_p(x, y) \mid y \in A_2 \right\}, \text{ for all } x \in A_1.
    \end{equation}
    The \emph{$c_p$-subdifferential} of the function $\phi$ is the set $\partial_{c_p}\varphi$ defined as
    \begin{align*}
    	\partial_{c_p}\varphi :={}& \left\{ (x,y) \in (A_1 \times A_2) \cap \H^2_\le \mid \varphi^{(c_p)}(y) - \varphi(x) = c_p(x, y) \right\} \\
    	={}& \left\{ (x,y) \in (A_1 \times A_2) \cap \H^2_\le \mid \forall z \in A_1, \ \phi(x) + c_p(x, y) \leq \phi(z) + c_p(x, z) \right\}.
    \end{align*}
    We say that a pair of functions $(u, v)$ are \emph{$c_p$-conjugate (relative to $(A_1, A_2)$)} if $u : A_1 \to \R$, $v : A_2 \to \R \cup \{+\infty\}$ and $v = u^{(c_p)}$. 
\end{definition}

Let us prove some useful properties about $c_p$-concave functions.

\begin{lemma}
	\label{lemma:concaveiff}
	Let $A_1$ and $A_2$ be non-empty subsets of $\H$. For any $\varphi: A_1 \to \R$, we have
	\begin{equation}
		\label{eq:tauptransforminequality}
		\inf_{y \in A_2}  \left(\phi^{(c_p)}(y) - c_p(x, y)\right) \geq \phi(x), \ \text{ for all } x \in A_1,
	\end{equation}
	and \cref{eq:tauptransforminequality} is an equality if and only if $\phi$ is $c_p$-concave relative to $(A_1, A_2)$. The $c_p$-transform $\phi^{(c_p)}$ of $\phi$ is bounded from below.
\end{lemma}

\begin{proof}
	For $x \in A_1$, we can write that
\begin{align*}
    \inf_{y \in A_2}  \left(\phi^{(c_p)}(y) - c_p(x, y)\right)  ={}& \inf_{y \in A_2} \sup_{x' \in A_1} \Big(\varphi(x') + c_p(x', y) - c_p(x, y)\Big) \\
    \geq{}& \inf_{y \in A_2} \Big(\varphi(x) + \tau(x,y)^p - c_p(x, y)\Big) = \phi(x),
\end{align*}	
where the inequality is obtained by taking $x' = x$ in the supremum. 

On the one hand, when \cref{eq:tauptransforminequality} is an equality for all $x \in A_1$, the function $\phi$ is clearly $c_p$-concave, simply by taking $\psi = \phi^{c_p}$ in \cref{def:subdifferential}. One the other hand, if $\phi$ is $c_p$-concave, then
\begin{align*}
    \inf_{y \in A_2}  \left(\phi^{(c_p)}(y) - c_p(x, y)\right)  ={}& \inf_{y \in A_2} \sup_{x' \in A_1} \Big(\varphi(x') + c_p(x', y) - c_p(x, y)\Big) \\
    ={}& \inf_{y \in A_2} \sup_{x' \in A_1} \inf_{y' \in A_2} \Big(\psi(y') - c_p(x', y') + c_p(x', y) - c_p(x, y)\Big)\\
    \leq{}& \inf_{y \in A_2} \sup_{x' \in A_1} \Big(\psi(y) - c_p(x', y) + c_p(x', y) - c_p(x, y)\Big)\\
    ={}& \inf_{y \in A_2} \Big(\psi(y) - c_p(x, y)\Big) = \phi(x),
\end{align*}
where the inequality is obtained by taking $y' = y$ in the infimum.

Finally, suppose that $\phi^{(c_p)}$ is not bounded from below, i.e. there is a sequence $(y_n)_{n \in \N} \subseteq A_2$ with $\phi^{(c_p)}(y_n) \neq +\infty$ and $\lim_{n \to +\infty} \phi^{(c_p)}(y_n) = - \infty$. Then, we have that
	\begin{equation}
		\label{eq:phitaupbounded}
		\phi(x) \leq{} \inf_{y \in A_2}  \left(\phi^{(c_p)}(y) - c_p(x, y)\right) \leq \phi^{(c_p)}(y_n) - c_p(x, y_n) \leq \phi^{(c_p)}(y_n),
	\end{equation}
	and thus, by taking the limit $n \to +\infty$, we obtain $\phi(x) \leq -\infty$, which is impossible since $\phi$ is real-valued.
\end{proof}

\begin{remark}
	\label{remark:phitaupnotinfinity}
	The $c_p$-transform of a $c_p$-concave function $\phi$ cannot be identically $+\infty$. Indeed, if that were the case, we would have $\phi(x) = +\infty$ for all $x \in A_1$ because of the equality in \cref{eq:tauptransforminequality} but, by definition, a $c_p$-concave function is real-valued.
\end{remark}

When the sets $A_1$ or $A_2$ are bounded, a $c_p$-concave function $\phi$ and its $c_p$-transform enjoy even more regularity properties.

\begin{lemma}
	Let $A_1$ and $A_2$ be non-empty subsets of $\H$ and $(\phi, \phi^{(c_p)})$ be a pair of $c_p$-conjugate functions relative to $(A_1, A_2)$.
	\begin{theoenum}
		\item If $A_2$ is bounded, then $\phi^{(c_p)}$ is bounded from above, and $\phi$ is continuous. 
        \item If $A_1$ is bounded, then $\phi$ is bounded from above, $\phi^{(c_p)}$ is real-valued and continuous.
    \end{theoenum}
\end{lemma}

\begin{proof}
	Let us start by proving (i). Assume that $A_2$ is bounded and suppose, for a contradiction, that there is a sequence $(y_n)_{n \in \N} \subseteq A_2$ with $\lim_{n \to +\infty} \phi^{(c_p)}(y_n) = +\infty$. For any $x \in A_1$, it holds that
	\[
	\phi^{(c_p)}(y_n) \leq \varphi(x) + c_p(x, y_n).
	\]
	Up to taking a subsequence, we can assume that $y_n \to y \in \mathrm{cl}(A_2)$ when $n \to +\infty$. By taking the limit above and recalling that $\tau$ is continuous by \cref{theorem:propertytau}, we obtain $\phi(x) = +\infty$ for all $x \in A_1$. This is impossible because $\phi$ is real-valued.

We now prove that $\phi$ is continuous. Let $q \in A_1$ and $(q_n)_{n \in \N} \subseteq A_1$ be a sequence converging to $q$. \cref{remark:phitaupnotinfinity} clarified the fact that $\phi^{(c_p)}$ is bounded from below and so the infimum property in \cref{eq:taupconcave} implies that for all $n$ sufficiently large, there exists $y_n \in A_2$ with
\begin{equation}
    \label{eq:proofinf1}
    \varphi(q_n) \geq \phi^{(c_p)}(y_n) - c_p(q_n, y_n) - \frac1n.
\end{equation}
By compactness of $\mathrm{cl}(A_2)$, we can extract a subsequence of $(y_n)_{n \in \N}$, which we can assume without loss of generality to be $(y_n)_{n \in \N}$, that converges to some $y \in \mathrm{cl}(A_2)$. We then have that for all $n \in \N$
\begin{align*}
    \phi(q) \leq{}& \phi^{(c_p)}(y_n) - c_p(q, y_n) = \phi^{(c_p)}(y_n) - c_p(q, y_n) + c_p(q_n, y_n) - c_p(q_n, y_n)\\
    \leq{}&\phi^{(c_p)}(y_n) - c_p(q, y_n) + c_p(q_n, y_n) + \phi(q_n) - \phi^{(c_p)}(y_n) + \frac{1}{n},
\end{align*}
the first inequality being justified by the infimum in \cref{eq:taupconcave} and the second by \cref{eq:proofinf1}. Simplifying and rearranging gives
\begin{equation}
    \label{eq:proofinf2}
    \phi(q_n) \geq \phi(q) + c_p(q, y_n) - c_p(q_n, y_n) - \frac{1}{n}.
\end{equation}
Similarly, for all $n$ sufficiently large, there exists $z_n \in A_2$ such that
\begin{equation}
    \label{eq:proofinf3}
    \varphi(q) \geq \phi^{(c_p)}(z_n) - c_p(q, z_n) - \frac1n.
\end{equation}
By compactness of $\mathrm{cl}(A_2)$ again, we can extract a subsequence of $(z_n)_{n \in \N}$, which we can assume without loss of generality to be $(z_n)_{n \in \N}$, that converges to some $z \in \mathrm{cl}(A_2)$. We then have that for all $n \in \N$
\begin{align*}
    \phi(q_n) \leq{}& \phi^{(c_p)}(z_n) - c_p(q_n, z_n) = \phi^{(c_p)}(z_n) - c_p(q_n, z_n) + c_p(q, z_n) - c_p(q, z_n)\\
    \leq{}&\phi^{(c_p)}(z_n) - c_p(q_n, z_n) + c_p(q, z_n) + \phi(q) - \phi^{(c_p)}(z_n) + \frac{1}{n}\\
    ={}& \phi(q) + c_p(q, z_n) - c_p(q_n, z_n) + \frac{1}{n},
\end{align*}
the first inequality being justified by the infimum in \cref{eq:taupconcave} and the second by \cref{eq:proofinf3}. Since $\tau$ is continuous in both variables by \cref{theorem:propertytau}, we deduce that $\phi(q_n)$ tends to $\phi(q)$ by letting $n \to +\infty$ in the inequality just above and in \cref{eq:proofinf2}.

In the proof of (ii), the arguments showing that $\phi$ is bounded from above and $\phi^{(c_p)}$ is continuous are similar to that of (i). The map $\phi^{(c_p)}$ is real-valued because
	\[
	\phi^{(c_p)}(y) = \sup_{x \in A_1} \left( \varphi(x) + c_p(x, y) \right) \leq \sup_{x \in A_1}  \varphi(x) + \sup_{x \in A_1} c_p(x, y) < +\infty,
	\]
	for all $y \in A_2$.
\end{proof}

In order to state and prove the next result, we will need to fix some notation. Given $A_1$ and $A_2$ two non-empty subsets of $\H$, $(\phi, \phi^{(c_p)})$ a pair of $c_p$-conjugate functions relative to $(A_1, A_2)$, and $R > 0$, we set
\begin{equation}
	\label{eq:AR}
	A_1^R := \{x \in \H \mid \diff_{\mathrm{sR}}(x, A_1) < R \}, \text{ and } A_2^R := \{x \in \H \mid \diff_{\mathrm{sR}}(x, A_2) < R \}.
\end{equation}
We also define
\begin{equation}
	\label{eq:phir}
	\phi_R(x) := \inf_{y \in A_2} \Big(\phi^{(c_p)}(y) - c_p(x, y)\Big), \ \text{ for } x \in A_1^R,
\end{equation}
and
\begin{equation}
	\label{eq:phitaupr}
	\phi^{(c_p)}_R (y) := \sup_{x \in A_1} \Big( \varphi(x) + c_p(x, y) \Big), \ \text{ for } y \in A_2^R.
\end{equation}

\begin{lemma}
	\label{lemma:phiRphicpR}
	Let $R > 0$, $A_1$ and $A_2$ be non-empty subsets of $\H$, and $(\phi, \phi^{(c_p)})$ be a pair of $c_p$-conjugate functions relative to $(A_1, A_2)$.
	\begin{theoenum}
		\item If $A_2$ is bounded, then $\phi_R$ is real-valued and $(\phi_R, \phi^{(c_p)})$ is a pair of $c_p$-conjugate functions relative to $(A_1^R, A_2)$.
		\item If $A_1$ is bounded, then $\phi^{(c_p)}_R$ is real-valued and $(\phi, \phi^{(c_p)}_R)$ is a pair of $c_p$-conjugate functions relative to $(A_1, A_2^R)$. 
		\item If $A_1$ and $A_2$ are bounded, then both $\phi_R$ and $\phi^{(c_p)}_R$ are real-valued and $(\phi_R, \phi^{(c_p)}_R)$ is a pair of $c_p$-conjugate functions relative to $(A_1^R, A_2^R)$. 
    \end{theoenum}
\end{lemma}

\begin{proof}
	When $A_2$ is bounded, then any $x \in A_1^R$ has
	\[
	\phi_R(x) \geq \inf_{y \in A_2} \phi^{(c_p)}(y) - \sup_{y \in \mathrm{cl}(A_2)} c_p(x, y) > - \infty
	\]
	since $\phi^{(c_p)}$ is bounded from below on $A_2$ by  \cref{lemma:concaveiff} and $\tau$ is continuous by \cref{theorem:propertytau}. This shows that $\phi_R$ is real-valued and the proof of (i) is concluded with  \cref{eq:phir} and \cref{lemma:concaveiff}.
	
	A similar argument shows that $\phi^{(c_p)}_R$ is real-valued if $A_1$ is bounded. In that case, we surely have
	\[
	\phi(x) = \inf_{y \in A_2} \Big( \phi^{(c_p)}_r(y) - c_p(x, y) \Big) \geq \inf_{y \in A_2^R} \Big( \phi^{(c_p)}_r(y) - c_p(x, y) \Big) 
	\]
	since $\phi^{(c_p)}_R = \phi^{(c_p)}$ on $A_2$. Note that by construction $\phi^{(c_p)}_R$ is the $c_p$-transform of $\phi$ relative to $(A_1, A_2^R)$ and so the other inequality follows from \cref{lemma:concaveiff}.
	
	Proving (iii) is just a combination of the reasonings above.
\end{proof}

We end this discussion by relating $c_p$-cyclical monotonicity and Kantorovich duality.

\begin{theorem}
\label{theorem:cCMiff}
	A Borel set $\Gamma \subseteq \H_\leq^2$ is $c_p$-cyclically monotone if and only if there exists a measurable map $\phi : P_1(\Gamma) \to \R$ that is $c_p$-concave relative to $(P_1(\Gamma), P_2(\Gamma))$ and such that $\Gamma \subseteq \partial_{c_p} \phi$.
\end{theorem}

\begin{proof}
	The necessary part of the statement is the difficult one, and it is shown in the proof of \cite[Theorem 2.26]{CavallettiMondino2023}. The converse, however, is proven straightforwardly by directly
verifying that \cref{def:cyclicallymonotone} is satisfied. Indeed, if a Borel set $\Gamma \subseteq \H_\leq^2$ is contained in the $c_p$-subdifferential $\partial_{c_p} \phi$ of some measurable $c_p$-concave map $\phi$, then it must be $c_p$-cyclically monotone because
	\begin{align*}
		\sum_{i=1}^N c_p(x_i, y_i) ={}&  \sum_{i=1}^N \phi^{(c_p)}(y_i) - \phi(x_i) = \sum_{i=1}^N \phi^{(c_p)}(y_i) - \phi(x_{i + 1})
		\geq \sum_{i=1}^N c_p(x_{i+1}, y_i),
	\end{align*}
	for all $N\in\mathbb{N}$ and all $(x_1, y_1),\ldots,(x_N, y_N)\in \Gamma$
    such that $x_{N+1}=x_1$. The first equality holds because $\Gamma \subseteq \partial_{c_p} \phi$. The second equality is verified just because the sum if finite and we can freely rearrange the terms. The inequality is valid because $\phi^{c_p}(y_i) \geq \phi(x_{i + 1}) + c_p(x_{i+1}, y_i)$ by definition of the $c_p$-transform.
\end{proof}

With \cref{theorem:cCMiff} at hand, finding a pair of functions $(u, v)$ that attains the maximum in
\cref{eq:dualproblem} is not difficult, as shown in the last part of the proof of \cite[Theorem 2.26]{CavallettiMondino2023} which we reproduce in our context for the sake of clarity.

\begin{theorem}
	\label{theorem:strongkantorovich}
	Let $(\mu, \nu) \in \mathcal{P}(\H^2)$ such that $\ell_p(\mu, \nu) < +\infty$. If a plan $\pi \in \Pi_{\leq}(\mu, \nu)$ is $c_p$-cyclically monotone, then it is an optimal transport plan from $\mu$ to $\nu$ and
	\begin{equation}
	\label{eq:strongKantorovichDuality}
		\ell_p(\mu, \nu)^p = \int_\H \varphi^{(c_p)}(y) \diff \nu(y) - \int_\H \varphi(x) \diff \mu(x),
	\end{equation}
	where $\phi : P_1(\Gamma) \to \R$ is any $c_p$-concave map relative to $(P_1(\Gamma), P_2(\Gamma))$, provided by \cref{theorem:cCMiff}, corresponding to a Borel $c_p$-cyclically monotone set $\Gamma \subseteq \H_\leq^2$ such that $\pi(\Gamma) = 1$. Furthermore, strong Kantorovich duality holds for $(\mu, \nu)$.
\end{theorem}

\begin{proof}
	Note that, since $\pi(\Gamma) = 1$ and $\pi$ has marginals $\mu$ and $\nu$, it follows that $\mu(P_1(\Gamma)) = 1$ and $\nu(P_2(\Gamma)) = 1$.
	A $c_p$-concave map $\phi : P_1(\Gamma) \to \R$ from \cref{theorem:cCMiff} verifies that $\Gamma \subseteq \partial_{c_p} \phi$. Thus, $c_p(x, y) = \phi^{(c_p)}(y) - \phi(x)$ for all $(x, y) \in \Gamma$, and we have
	\begin{align*}
		K(\pi)={}& \int_\Gamma c_p(x, y) \diff \pi(x, y) = \int_{P_2(\Gamma)} \varphi^{(c_p)}(y) \diff \nu(y) - \int_{P_1(\Gamma)} \varphi(x) \diff \mu(x)\\
		={}& \int_\H \varphi^{(c_p)}(y) \diff \nu(y) - \int_\H \varphi(x) \diff \mu(x) \leq \sup_{\pi' \in \Pi_{\leq}(\mu, \nu)} K(\pi') =: \ell_p(\mu, \nu)^p.
	\end{align*}
	The functions $\phi$ and $\phi^{(c_p)}$ are in $\mathrm{L}^1(\mu)$ and $\mathrm{L}^1(\nu)$ respectively, from the assumption $\ell_p(\mu, \nu) < +\infty$.
	By definition of the $c_p$-transform, we also have $\phi^{(c_p)}(y) - \phi(x) \geq c_p(x, y)$ for $(\mu \otimes \nu)$-every $(x, y) \in \H$ and integrating this inequality with respect to any $\pi' \in \Pi_{\leq}(\mu, \nu)$ yields
	\[
	\ell_p(\mu, \nu)^p \leq \int_\H \varphi^{(c_p)}(y) \diff \nu(y) - \int_\H \varphi(x) \diff \mu(x).
	\]
	Since we always have
	\[
	\ell_p(\mu, \nu)^p \leq \inf D(u, v) \text{, and } \inf D(u, v) \leq \int_\H \varphi^{(c_p)}(y) \diff \nu(y) - \int_\H \varphi(x) \diff \mu(x),
	\]
	we deduce that strong Kantorovich duality holds.
\end{proof}

A measurable $c_p$-concave map satisfying \cref{eq:strongKantorovichDuality} is called a $c_p$-Kantorovich potential.

 \subsection{Brenier's theorem in the sub-Lorentzian Heisenberg group}
 
 \label{subsection:brenierSLH}

We are now progressing toward the proof of Brenier's theorem, and we are going to work through some preparatory lemmas first.

\begin{lemma}
	\label{lemma:subdifferentialisagraph}
	Let $A_1$ and $A_2$ be non-empty subsets of $\H$ such that $A_1 \times A_2 \subseteq \H_{\ll}^2$, and $\phi : A_1 \to \R$ is a $c_p$-concave map relative to $(A_1, A_2)$. If $A_1$ is open and $\phi$ is differentiable almost everywhere, then $\mathcal{L}^3(P_1(\partial_{c_p} \phi \setminus \Gamma)) = 0$, where $\Gamma \subseteq \partial_{c_p}\phi$ is the measurable set given by
	\begin{equation}
		\label{eq:Gammagraph}
		\Gamma := \{ (q, T(q)) \mid q \in S \}, \text{  where } S := \{ q \in P_1(\partial_{c_p}\phi) \mid \phi \text{ is differentiable at } q \}
	\end{equation}
	and $T : S \to \H$ is the map
	\[
	T(q) := \exp_q\left( - \Diff_{q} \phi \bigg/ {\left(\sqrt{2H(q, \Diff_q \phi)}\right)^{\tfrac{p - 2}{p-1}}}\right).
	\]  
\end{lemma}

\begin{proof}
	We have
	\begin{align*}
		P_1(\partial_{c_p}\phi) \setminus S ={}& \{ q \in P_1(\partial_{c_p}\phi) \mid \phi \text{ is not differentiable at } q \} \\
		\subseteq{}& \{ q \in A_1 \mid \phi \text{ is not differentiable at } q \},
	\end{align*}
	and therefore $\mathcal{L}^3(P_1(\partial_{c_p}\phi) \setminus S) = 0$. If $q \in S$, then $\phi$ is differentiable at $q$ and there exists $q' \in A_2 \subseteq I^+(q)$ such that $(q, q') \in \partial_{c_p} \phi$. In particular, we have $\varphi^{(c_p)}(q') - \varphi(q) = c_p(q, q')$ and
\begin{equation}
\label{eq:keyinequality}
\varphi^{(c_p)}(q') - \varphi(x) \geq c_p(x, q') = \frac{1}{p} \tau(x, q')^p, \quad \text{for all $x \in A_1$.}
\end{equation}
    We can then apply      \cref{lemma:Geometric} to the function $\psi(x) := \phi^{(c_p)}(q') - \phi(x)$, and we obtain \cref{eq:q'=T(q)}, which we write as $q' = T(q)$. Then $\Gamma = \partial_{c_p} \phi \cap P_1^{-1}(S)$ and thus
    \begin{align*}
    	P_1(\partial_{c_p} \phi \setminus \Gamma) ={}& P_1(\partial_{c_p} \phi \setminus P_1^{-1}(S)) = P_1(\partial_{c_p}\phi) \setminus S
    \end{align*}
    which implies $\mathcal{L}^3(P_1(\partial_{c_p} \phi \setminus \Gamma)) = 0$.
\end{proof}

Similarly, we can prove using \cref{lemma:Geometric2}.

\begin{lemma}
	\label{lemma:subdifferentialisagraph2}
	Let $A_1$ and $A_2$ be non-empty subsets of $\H$ such that $A_1 \times A_2 \subseteq \H_{\ll}^2$, and $\phi : A_1 \to \R$ is a $c_p$-concave map relative to $(A_1, A_2)$. If $A_2$ is open and $\phi^{(c_p)}$ is differentiable almost everywhere, then $\mathcal{L}^3(P_2(\partial_{c_p} \phi \setminus \Gamma)) = 0$, where $\Gamma \subseteq \partial_{c_p}\phi$ is the measurable set given by
	\begin{equation}
		\label{eq:Gammagraph2}
		\Gamma := \{ (T(q), q) \mid q \in S \}, \text{  where } S := \{ q \in P_2(\partial_{c_p}\phi) \mid \phi^{(c_p)} \text{ is differentiable at } q \}
	\end{equation}
	and $T : S \to \H$ is the map
	\[
	T(q) = \exp_q\left(\Diff_{q} \phi^{(c_p)} \bigg/ {\left(\sqrt{2H(q, \Diff_q \phi^{(c_p)})}\right)^{\tfrac{p - 2}{p-1}}}\right).
	\]  
\end{lemma}

\begin{remark}
	\label{remark:subdifferentialisagraph}
	We can replace differentiable almost everywhere with $\mu$-almost everywhere and then $\mu(P_2(\partial_{c_p} \phi \setminus \Gamma)) = 0$ with the obvious modifications.
\end{remark}

The next lemmas addresses further regularity properties of $c_p$-concave maps and of their $c_p$-transform. The notions of functions of horizontal bounded variation is taken from \cite{Ambrosio2003}.

\begin{lemma}
	\label{lemma:locallysemiconcaveBVH}
    Let $A_1$ and $A_2$ be non-empty subsets of $\H$ such that $A_1 \times A_2 \subseteq \H_{\ll}^2$, and let $\phi : A_1 \to \R$ be a $c_p$-concave map relative to $(A_1, A_2)$.
    \begin{theoenum}
		\item If $A_2$ is bounded, then $\phi$ is locally semiconcave and belongs to $\mathrm{BV}_{\mathrm{loc}, H}^2(A_1)$.
		\item If $A_1$ is bounded, then $\phi^{(c_p)}$ is locally semiconvex and belongs to $\mathrm{BV}_{\mathrm{loc}, H}^2(A_2)$.
    \end{theoenum}
\end{lemma}

\begin{proof}
	We only provide the proof for (i) as the proof for (ii) is similar. Since $A_2$ is bounded, the family of functions $x \in A_1 \mapsto c_p(x, y)$ is locally semiconvex with semiconvexity constants uniformly bounded in $y \in A_2$, thanks to \cref{theorem:propertytau}. Because we can express $\phi$ as the infimum in \cref{lemma:concaveiff}, the argument of \cite[Proposition 2.1.5, Corollary 2.1.6]{Cannarsa2004} establishes that $\phi$ is locally semiconcave too. The fact that $\phi \in \mathrm{BV}_{\mathrm{loc}, H}^2(A_1)$ follows from the same proof as in \cite[Theorem 4.3]{Ambrosio2003}. Indeed, by \cref{theorem:propertytau}, we know that $\tau$ is analytic on $A_1 \times A_2 \subseteq \H^2_\ll$ and thus
	\[
	\sup_{\substack{x \in \Omega \\ y \in A_2}} \left\{ |XXc_p(x, y)|, |XYc_p(x, y)|, |YYc_p(x, y)| \right\} < +\infty,
	\]
	for every bounded subset $\Omega \subseteq A_1$.
\end{proof}

Brenier's theorem in the sub-Lorentzian Heisenberg group—specifically, the existence and uniqueness of the optimal transport plan between probability measures—finally follows.

\begin{theorem}
	\label{thm:Brenier}
    Let $\mu, \nu \in \mathcal{P}(\H)$ with $\mu \ll \mathcal{L}^3$ and $\mathrm{supp}(\mu) \times \mathrm{supp}(\nu) \subseteq \H^2_\ll$. Assume that $\mathrm{supp}(\nu)$ is compact and that the condition \cref{remark:condition(ii)} is satisfied. Then, there exists a unique optimal transport plan $\pi$ from $\mu$ to $\nu$, i.e., a unique solution to \cref{eq:LorentzKantorovich}. Furthermore, $\pi$ is induced by a unique (forward) transport map $T$, which is the unique solution to \cref{eq:LorentzMonge} and is given by 
\begin{equation}
	\label{eq:brenierexponential}
	T(q) = \exp_q\left({-\Diff_{q} \phi} \bigg/ {\left(\sqrt{2H(q, \Diff_q \phi)}\right)^{\tfrac{p - 2}{p-1}}}\right),
\end{equation}
where $\phi : A_1 \to \R$ is any map that is $c_p$-concave relative to a pair of sets $(A_1, A_2)$, with $\mu(A_1) = \nu(A_2) = 1$, $A_1$ being open, $A_2$ bounded, $\mathrm{supp}(\mu) \subseteq A_1$, and $A_1 \times A_2 \subseteq \H_\ll^2$.
\end{theorem}

\begin{proof}
     Since $\mathrm{supp}(\mu) \times \mathrm{supp}(\nu) \subseteq \H^2_\ll$, we have $\mu \otimes \nu \in \Pi_{\ll}(\mu, \nu) \subseteq \Pi_{\leq}(\mu, \nu)$ and thus the conditions \cref{remark:condition(i)} and \cref{remark:condition(ii)} are both satisfied. Consequently, we know by \cref{prop:optimalexistence} that $\ell_p(\mu, \nu) < +\infty$ and that there is an optimal transport plan $\pi \in \Pi_{\leq}(\mu, \nu)$ from $\mu$ to $\nu$. Any such plan must be an element of $\Pi_{\ll}(\mu, \nu)$ because $\mathrm{supp}(\pi) \subseteq \mathrm{supp}(\mu) \times \mathrm{supp}(\nu)$.
     
     Let $\pi \in \Pi_{\leq}(\mu, \nu)$ denote any optimal plan from $\mu$ to $\nu$, which, by \cref{prop:CMandOptimality} (i), must necessarily be $c_p$-cyclically monotone which means that its support $\mathrm{supp}(\pi)$ is $c_p$-cyclically monotone. By virtue of \cref{theorem:cCMiff}, there exists a measurable $c_p$-concave map $\phi$ relative to $(A_1, A_2) := (P_1(\mathrm{supp}(\pi)), P_2(\mathrm{supp}(\pi)))$ such that $\mu(A_1) = \nu(A_2) = 1$ and $\mathrm{supp}(\pi) \subseteq \partial_{\tau^p} \phi$. Consequently, the conclusions of \cref{theorem:strongkantorovich}, including strong Kantorovich duality \cref{eq:strongKantorovichDuality}, are valid.
We have $A_1 \subseteq \mathrm{supp}(\mu)$ and $A_2 \subseteq \mathrm{supp}(\nu)$; hence, $A_1 \times A_2 \subseteq \H_{\ll}^2$, and $A_2$ is bounded due to the assumption on $\mathrm{supp}(\nu)$. Using \cref{lemma:phiRphicpR} (i), we obtain a pair of $c_p$-concave maps $(\phi_R, \phi^{(c_p)})$ relative to $(A_1^R, A_2)$, following the notation established in \cref{eq:AR}, \cref{eq:phir}, and \cref{eq:phitaupr}. For sufficiently small $R > 0$, we still have $A_1^R \times A_2 \subseteq \H_{\ll}^2$, and $A_1^R$ is a compact neighbourhood of $A_1$ and thus of $\mathrm{supp}(\mu)$ too. This shows that there exists at least one $c_p$-concave map satisfying all the properties listed just after \cref{eq:brenierexponential}.
	
	Now, with a slight abuse of notation, let $\phi : A_1 \to \R$ be any map that is $c_p$-concave relative to a pair of sets $(A_1, A_2)$ and that satisfies those properties. By \cref{lemma:locallysemiconcaveBVH}, the map $\phi$ is locally semiconcave. In particular, \cite[Theorem 2.1.7]{Cannarsa2004} implies that $\phi$ is locally Lipschitz and therefore differentiable almost everywhere in the interior of $A_1$.
 	
We now apply \cref{lemma:subdifferentialisagraph} to the pair $(\phi, \phi^{(c_p)})$, obtaining the set $\Gamma$ as described in \cref{eq:Gammagraph}. Since it is assumed that $\mu \ll \mathcal{L}^3$, we conclude that $\mu(P_1(\partial_{\tau^p} \phi \setminus \Gamma)) = 0$, i.e., for $\mu$-almost every $q \in A_1$, there exists a unique $q' \in A_2$ such that $(q, q') \in \partial_{c_p} \phi$. The first marginal of $\pi$ is $\mu$, so we have 
\begin{equation}
\label{eq:keystepinbrenier}
\mu(P_1(\partial_{c_p} \phi \setminus \Gamma)) = \pi(P_1^{-1} \circ P_1(\partial_{c_p} \phi \setminus \Gamma)) \geq \pi(\partial_{c_p} \phi \setminus \Gamma).
\end{equation}
In particular, $\pi(\Gamma) = 1$, and $\pi = (\mathrm{Id} \times T)_\sharp \mu$ for the function $T$ given in \cref{eq:brenierexponential}.

    It remains to show the uniqueness part of the statement, which is standard. Suppose that $\pi$ and $\pi'$ are optimal transport plan from $\mu$ to $\nu$, then from the discussion above there exist maps $T$ and $T'$ such that $\pi = (\mathrm{Id} \times T)_\sharp \mu$ and $\pi' = (\mathrm{Id} \times T')_\sharp \mu$. The plan $\pi'' := \tfrac{1}{2}(\pi + \pi')$ has cost $K(\pi'') = K(\pi) = K(\pi')$ and thus is also optimal. Therefore, we must have $\pi'' = (\mathrm{Id} \times T'')_\sharp \mu$ for some map $T''$ too and thus $T(q) = T'(q)$ for $\mu$-almost every $q \in \H$. 
\end{proof}

A converse to \cref{thm:Brenier} can also be proven.

\begin{theorem}
\label{thm:conversebrenier}
Let $\mu \in \mathcal{P}(\H)$, and let $A_1$ and $A_2$ be non-empty subsets of $\H$ such that $A_1 \times A_2 \subseteq \H_{\ll}^2$. Suppose $\phi : A_1 \to \R$ is a $c_p$-concave map relative to $(A_1, A_2)$, where $A_1$ is open and $\phi$ is differentiable $\mu$-almost everywhere. 
If $\pi = (\mathrm{Id} \times T)_\sharp \mu$, where $T$ is the map given by \cref{lemma:subdifferentialisagraph}, then $\pi$ is an optimal transport plan between $\mu$ and $\nu := T_\sharp \mu$, provided the condition \cref{remark:condition(ii)} is satisfied for the pair $(\mu, \nu)$.
\end{theorem}

\begin{proof}
From \cref{remark:subdifferentialisagraph}, we know that $\mu(P_1(\partial_{c_p} \phi \setminus \Gamma)) = 0$ and that the set $\Gamma$ given in \cref{eq:Gammagraph}. satisfies $\Gamma \subseteq \partial_{c_p}\phi$. The plan $\pi := (\mathrm{Id} \times T)_\sharp \mu$ has first marginal $\mu$, and by reasoning as in \cref{eq:keystepinbrenier}, this implies $\pi(\Gamma) = 1$. Since $\Gamma \subseteq \partial_{c_p}\phi$, it follows that $\pi(\partial_{c_p}\phi) = 1$, establishing that $\pi$ is $c_p$-cyclically monotone thanks to \cref{theorem:cCMiff}. Furthermore, $\mathrm{supp}(\mu) \subseteq A_1$ and $\mathrm{supp}(\nu) \subseteq A_2$, so $\mathrm{supp}(\mu) \times \mathrm{supp}(\nu) \subseteq \H_\ll^2$. Both conditions \cref{remark:condition(i)} and \cref{remark:condition(ii)} are satisfied, allowing us to apply \cref{prop:CMandOptimality} (ii) and conclude that $\pi$ is optimal.
\end{proof}

Analogous statements can be established for the backward Lorentz-Monge transportation problem \cref{eq:LorentzMonge-}. The proofs are entirely similar, with the main differences being the uses of \cref{lemma:phiRphicpR} (i) and \cref{lemma:subdifferentialisagraph}, which are replaced by \cref{lemma:phiRphicpR} (ii) and \cref{lemma:subdifferentialisagraph2}, respectively. We state these results here for the sake of completeness.

\begin{theorem}
	\label{thm:Brenier2}
    Let $\mu, \nu \in \mathcal{P}(\H)$ with $\nu \ll \mathcal{L}^3$ and $\mathrm{supp}(\mu) \times \mathrm{supp}(\nu) \subseteq \H^2_\ll$. Assume that $\mathrm{supp}(\mu)$ is compact and that the condition \cref{remark:condition(ii)} is satisfied. Then, there exists a unique optimal transport plan $\pi$ from $\mu$ to $\nu$, i.e., a unique solution to \cref{eq:LorentzKantorovich}. Furthermore, $\pi$ is induced by a unique (backward) transport map $T$, which is the unique solution to \cref{eq:LorentzMonge-} and is given by 
\begin{equation}
\label{eq:brenierexponential2}
	T(q) = \exp_q\left(\Diff_{q} \phi^{(c_p)} \bigg/ {\left(\sqrt{2H(q, \Diff_q \phi^{(c_p)})}\right)^{\tfrac{p - 2}{p-1}}}\right),
\end{equation}
where $\phi : A_1 \to \R$ is any map that is $c_p$-concave relative to a pair of sets $(A_1, A_2)$, with $\mu(A_1) = \nu(A_2) = 1$, $A_2$ being open, $A_1$ bounded, $\mathrm{supp}(\nu) \subseteq A_2$, and $A_1 \times A_2 \subseteq \H_\ll^2$.
\end{theorem}

\begin{theorem}
\label{thm:conversebrenier2}
Let $\nu \in \mathcal{P}(\H)$, and let $A_1$ and $A_2$ be non-empty subsets of $\H$ such that $A_1 \times A_2 \subseteq \H_{\ll}^2$. Suppose $\phi : A_1 \to \R$ is a $c_p$-concave map relative to $(A_1, A_2)$, where $A_2$ is open, $\phi^{(c_p)}$ is real-valued and differentiable $\nu$-almost everywhere. 
If $\pi = (T \times \mathrm{Id})_\sharp \nu$, where $T$ is the map given by \cref{lemma:subdifferentialisagraph2}, then $\pi$ is an optimal transport plan between $\mu := T_\sharp \nu$ and $\nu$, provided the condition \cref{remark:condition(ii)} is satisfied for the pair $(\mu, \nu)$.
\end{theorem}

\begin{remark}
\label{remark:Tt(q)geodesic}
	The following observation is made within the context of \cref{thm:Brenier} or \cref{thm:conversebrenier}. At a point $q \in \H$ where $\phi : A_1 \to \R$ is differentiable, the curve
	\[
	\interval{0}{1} \to \H : t \mapsto T_t(q) := \exp_q\left({-t\Diff_{q} \phi} \bigg/ {\left(\sqrt{2H(q, \Diff_q \phi)}\right)^{\tfrac{p - 2}{p-1}}}\right)
	\]
	is the unique maximising geodesic joining $q$ to $T(q)$ parametrised by constant speed on the interval $\interval{0}{1}$. In particular, any two admissible $c_p$-concave maps $\phi$ and $\tilde \phi$ inducing the optimal (forward) transport map \cref{eq:brenierexponential} must satisfy $\mathrm{D}_q \phi = \mathrm{D}_q \tilde \phi$ for $\mu$-almost every $q \in \H$. Similar reasoning can be made for \cref{thm:Brenier2} and \cref{thm:conversebrenier2}, the backward optimal transport map \cref{eq:brenierexponential2} and the function $\phi^{(c_p)}$ inducing it.
\end{remark}

Finally, when the conditions are met to ensure the existence of both the forward and backward optimal transport maps of \cref{thm:Brenier} and \cref{thm:Brenier2}, it can be shown that they are inverses of each other (almost everywhere).

\begin{theorem}
\label{theorem:inversetransportmap}
	Let $\mu, \nu \in \mathcal{P}_{\mathrm{c}}^{\mathrm{ac}}(\H, \mathcal{L}^3)$ with $\mathrm{supp}(\mu) \times \mathrm{supp}(\nu) \subseteq \H^2_\ll$. The forward optimal transport map $T^{\mu \to \nu}$ and the backward optimal transport map $T^{\nu \to \mu}$, given in \cref{eq:brenierexponential} and \cref{eq:brenierexponential2} respectively, are inverses. Specifically, $T^{\mu \to \nu} \circ T^{\nu \to \mu}(q) = q$ for $\nu$-almost every $q \in \H$ and $T^{\nu \to \mu} \circ T^{\mu \to \nu}(q) = q$ for $\mu$-almost every $q \in \H$.
\end{theorem}

\begin{proof}
	This follows directly from \cref{remark:Tt(q)geodesic}, and we argue only for $T^{\nu \to \mu} \circ T^{\mu \to \nu}$. For $\mu$-almost every $q \in \H$, there exists a unique maximising geodesic from $q$ to $T^{\mu \to \nu}(q)$, and $T^{\nu \to \mu}$ is defined at $T^{\mu \to \nu}(q)$. The map $T^{\nu \to \mu}$ returns the endpoint of this geodesic curve back to its starting point, ensuring that $T^{\nu \to \mu}(T^{\mu \to \nu}(q)) = q$.
\end{proof}

\begin{remark}
	\label{remark:expqtoexpe}
	The sub-Riemannian version of Brenier's theorem in \cite[Theorem 5.1]{AmbrosioRigot2004} is stated using the exponential map $\exp_e$, based at the identity element of $\H$, rather than the exponential map $\exp_q$, as we do in \cref{eq:brenierexponential}. We can do the same using left-invariance and \cref{eq:expqtoexpe}.
\end{remark}

\subsection{A sub-Lorentzian Monge-Ampère type equation and some examples}

\label{subsection:MongeAmpere+examples}

It is well known that Brenier's Theorem is closely linked to the Monge-Ampère equation. This can be observed by deriving the partial differential equation that arises from a change of variable on the condition $T_\sharp \mu_0 = \mu_1$. We follow some ideas presented in \cite[Section 5.1]{Barilari2019} for sub-Riemannian manifolds.
In the smooth Lorentzian setting, this was achieved in \cite[Corollary 5.9, Theorem 5.10, Corollary 5.11]{McCann2020}, but see also \cite[Corollary 4.20]{BraunOhta2024} for the Lorentz-Finsler case.

{
\begin{theorem}
	\label{thm:mongeampere}
	Under the same assumptions and notations of \cref{thm:Brenier} and \cref{remark:Tt(q)geodesic}, the map $T_t : \H \to \H$  is differentiable $\mu_0$-almost everywhere and $\det(\diff_q T_t) > 0$ for all $t \in \interval{0}{1}$ and $\mu_0$-almost every $q \in \H$. Moreover, there exists a unique $\ell_p$-geodesic from $\mu_0 := \mu$ to $\mu_1 := \nu$, given by $\mu_t = (T_t)_{\sharp}\mu_0$ for $t \in \interval{0}{1}$, and $\mu_t \ll \mathcal{L}^3$ for all $t \in \rinterval{0}{1}$. Finally, denoting by $\rho_t := \diff \mu_t / \diff \mathcal{L}^3$ for $t \in \rinterval{0}{1}$, it holds that
	\begin{equation}
		\label{eq:Monge-Ampere}
		\rho_0(q) = \rho_t(T_t(q)) \mathrm{det}(\diff_q T_t),
	\end{equation}
	for $\mu_0$-almost every $q \in \H$ and all $t \in \rinterval{0}{1}$. This conclusion also holds for $t = 1$ provided $\mu_1 \ll \mathcal{L}^3$.
	\end{theorem}

\begin{remark}
	The determinant of the linear map $\diff_q T_t : \T_q(\H) \to \T_{T_t(q)}(\H)$ in \cref{thm:mongeampere} is computed with respect to the exponential coordinates $(x, y, z)$.
\end{remark}

\begin{proof}
	The map $\phi : A_1 \to \R$ appearing in \cref{eq:brenierexponential} is twice differentiable $\mathcal{L}^3$-almost everywhere. Indeed, it is locally semiconcave by \cref{lemma:locallysemiconcaveBVH}, thus twice differentiable almost everywhere by Alexandrov's Theorem (see \cite[Theorem A.5]{Figalli2010}). Since $\mu_0 \ll \mathcal{L}^3$, we deduce that $T_t$ is differentiable $\mu_0$-almost everywhere. Furthermore, we have, for $0 \leq s \leq t \leq 1$, that
	\begin{align*}
		\ell_p(\mu_s, \mu_t)^p \geq{}& \int_{\H \times \H} c_p(T_s(x), T_t(x)) \mu_0(\diff x)\\
        ={}& (t - s)^p \int_{\H \times \H}c_p(x, T(x)) \mu_0(\diff x) = (t - s)^p \ell_p(\mu_0, \mu_1).
	\end{align*}
    The reverse triangle inequality gives
    \begin{align*}
        \ell_p(\mu_0,\mu_1) \geq{}& \ell_p(\mu_0,\mu_s) + \ell_p(\mu_s,\mu_t) + \ell_p(\mu_t,\mu_1) \\
        \geq{}& s \ell_p(\mu_0,\mu_1) + (t - s) \ell_p(\mu_0,\mu_1) + (1 - t) \ell_p(\mu_0,\mu_1) = \ell_p(\mu_0,\mu_1),
    \end{align*}
    and thus $\ell_p(\mu_s, \mu_t) = (t - s) \ell_p(\mu_0, \mu_1)$, proving that $(\mu_t)_{t \in \interval{0}{1}}$ is an $\ell_p$-geodesic according to \cref{eq:Wassersteingeodesic}. The fact that this geodesic is the unique one from $\mu_0$ to $\mu_1$ follows from the uniqueness statements in \cref{thm:Brenier}.

	We are now going to prove that $\mu_t \ll \mathcal{L}^3$, adapting the ideas in \cite[Section 6.3]{Figalli2010}. By the reverse triangle inequality and the concavity of the function $t \mapsto t^p$ for $p \in \ointerval{0}{1}$, we have, for all $q, q', q'' \in \H$ and $t \in \ointerval{0}{1}$, that
	\begin{equation}
		\label{eq:triangle+concave}
		c_p(q, q') \geq \frac{1}{p}(\tau(q, q'') + \tau(q'', q'))^p \geq \frac{c_p(q, q'')}{t^{p - 1}} + \frac{c_p(q'', q')}{(1 - t)^{p - 1}},
	\end{equation}
	with equality if and only if there exists a constant speed geodesic $\gamma : \interval{0}{1} \to \H$ such that $\gamma(0) = q$, $\gamma(1) = q'$ and $\gamma(t) = q''$. By the definition of $T_t$, we therefore have
	\begin{equation}
		\label{eq:cpTt}
		c_p(q, T(q)) = \frac{c_p(q, T_t(q))}{t^{p - 1}} + \frac{c_p(T_t(q), T(q))}{(1 - t)^{p - 1}},
	\end{equation}
	for all $t \in \interval{0}{1}$ and $\mu_0$-a.e. $q \in \H$. We now set the following function, given $q'' \in \H$,
	\[
	\phi_{t}(q'') := \sup_{q \in \mathrm{supp}(\mu)} \left( \phi(q) + \frac{c_p(q, q'')}{t^{p - 1}} \right)
	\]
	and
	\[
	\phi^{(c_p)}_{1 - t}(q'') := \inf_{q' \in A_2} \left( \phi^{(c_p)}(q') + \frac{c_p(q'', q')}{(1 - t)^{p - 1}} \right).
	\]
	Since $\phi^{(c_p)}(q') - \phi(q) \geq c_p(q, q')$ for all $(q, q') \in A_1 \times A_2$, we obtain $\phi^{(c_p)}_{1 - t}(q) - \phi_t(q) \geq 0$, using \cref{eq:triangle+concave}. By construction, we also have that $\phi^{(c_p)}(T(q)) - \phi(q) = c_p(q, T(q))$ for $\mu_0$-a.e. $q \in \H$ and therefore, with \cref{eq:cpTt}, we obtain
	\[
	\phi^{(c_p)}_{1 - t}(T_t(q)) - \phi(T_t(q)) = 0, \text{ for $t \in \ointerval{0}{1}$ and $\mu_0$-a.e. $q \in \H$}.
	\]
	This identity can be rewritten as
	\[
	\phi^{(c_p)}_{1 - t}(q) - \phi(q) = 0, \text{ for $t \in \ointerval{0}{1}$ and $\mu_t$-a.e. $q \in \H$}.
	\]
	By the properties of the time-separation function and the boundedness of $\mathrm{supp}(\mu)$ and $A_2$, we can once more infer good regularity properties on $\phi^{(c_p)}_{1 - t}$ and $\phi$: the former must be locally semiconvex and the latter locally semiconcave, in a neighbourhood of $T_t(A_1) = T_{1-t}(A_2)$. From \cite[Theorem A.19]{Figalli2010}, we deduce that the differentials $\mathrm{D}_q \phi^{(c_p)}_{1 - t}$ and $\mathrm{D}_q \phi$ exist, are equal for $\mu_0$-almost every $q \in \H$, and that the map $q \mapsto \mathrm{D}_q \phi^{(c_p)}_{1 - t} = \mathrm{D}_q \phi$ is locally Lipschitz (in charts) on $T_t(A_1)$. Given $q \in \mathrm{supp}(\mu)$, we have that
	\[
	\phi_t(q'') \geq \phi(q) + \frac{c_p(q, q'')}{t^{p - 1}}, \text{ i.e. } \psi_t(q''):= t^{p - 1}(\phi_t(q'') - \phi(q)) \geq c_p(q, q'')
	\]
	for all $q'' \in \H$, and with equality at $q'' = T_t(q)$ for $\mu_0$-a.e. $q \in \H$. \cref{lemma:Geometric2} yields that
	\[
	q = \exp_{T_t(q)}\left({\Diff_{T_t(q)} \psi_t} \bigg/ {\left(\sqrt{2H(T_t(q), \Diff_{T_t(q)} \psi_t)}\right)^{\tfrac{p - 2}{p-1}}}\right).
	\]
	Consequently, since the map
	\[
	F_t := q \mapsto \exp_{q}\left({\Diff_{q} \psi_t} \bigg/ {\left(\sqrt{2H(q, \Diff_{q} \psi_t)}\right)^{\tfrac{p - 2}{p-1}}}\right)
	\]
	is locally Lipschitz on $\mathrm{supp}(\mu_t) \cap T_t(A_1)$, we deduce that $\mu_t$ is absolutely continuous. Indeed, if that were not the case, $\mu_0 = (F_t)_\sharp \mu_t$ would also not be absolutely continuous, contradicting our assumption.

    In particular, $\mu_0, \mu_t \in \mathcal{P}_{\mathrm{c}}^{\mathrm{ac}}(\H, \mathcal{L}^3)$ with $\mathrm{supp}(\mu) \times \mathrm{supp}(\nu) \subseteq \H^2_\ll$ and the same arguments as in \cref{theorem:inversetransportmap} yield that $T_t$ is injective $\mu_0$-almost everywhere. It remains to apply the general area formula general found in \cite[Lemma 40]{Barilari2019}, and we obtain $\det(\diff_q T_t) > 0$ for all $t \in \interval{0}{1}$ and $\mu_0$-almost every $q \in \H$, as well as \cref{eq:Monge-Ampere}.
\end{proof}
}

Finally, we give some examples of optimal transport maps in the sub-Lorentzian Heisenberg group. We start with a simple example that is only based on geometric arguments.
We denote by $\mathsf{p} : \H \to \R^2$ the projection defined as $\mathsf{p}(x, y, z) = (x, y)$. The Minkowski plane consists of the Lorentzian structure on $\R^2$ induced from the metric
	\[
	\langle (u_1, u_2), (v_1, v_2) \rangle := u_2 v_2 - u_1 v_1.
	\]
	It is well-known that the time-separation function $\tilde \tau$ of the Minkowski plane is given by
	\[
	\qquad \tilde{\tau}(u, v) = \sqrt{- \langle u - v, u - v \rangle},
	\]
	for all $u, v \in \R^2$ such that $\langle u, v \rangle \leq 0$ and $\tilde \tau = 0$ otherwise. Note that the causal and chronological future/past of $(x, y)\in \R^2$ are well-known to be given by
	\begin{equation}
    \label{eq:tildeJ+}
    \tilde J^{\pm}(x, y) := \left\{ (x', y') \in \H | -(x -x')^2 + (y - y')^2 \leq 0, \pm x \geq 0 \right\}
	\end{equation}
	and
	\begin{equation}
    \label{eq:tildeI+}
    \tilde I^{\pm}(x, y) := \left\{ (x', y') \in \H | -(x - x')^2 + (y - y')^2 < 0, \pm x > 0   \right\}.
	\end{equation}
	One may consider the Monge-Kantorovich optimal transport problem in this setting with the cost $\tilde c_p := \tilde \tau^p/p$. The Minkowski analogue to the functional appearing in \cref{eq:LorentzMonge} will be written as $\tilde M^+$.

\begin{proposition}
\label{prop:projectionToptimal}
	Let $\mu_0, \mu_1 \in \mathcal{P}(\H)$ and $\tilde \mu_0, \tilde \mu_1 \in \mathcal{P}(\R^2)$. Suppose that $\mathsf{p}_\sharp \mu_0 = \tilde \mu_0$ and assume that $\tilde T = (\tilde T_1, \tilde T_2) : \R_2 \to \R_2$ is a (forward) transport map from $\tilde \mu_0$ to $\tilde \mu_1$ optimal for the Monge problem in the Minkowski plane equipped with the Lorentzian cost $\tilde c_p$. Then, the map $T : \H \to \H$ given by
	\[
	T(x, y, z) := (x, y, z) \cdot \exp_e(-(\tilde T_1(x, y)- x) \diff u + (\tilde T_2(x, y) - y) \diff v)
	\]
	is an optimal (forward) transport map from $\mu_0$ to $\mu_1$ for the Monge problem in $\H$ with cost $c_p$.
\end{proposition}

\begin{proof}
	Note that the assumption that $\tilde T$ is a (forward) transport map means that $\tilde T(x, y) \in \tilde J^+(x, y)$, i.e. $(\tilde T_1(x, y) - x)^2 - (\tilde T_2(x, y) - y)^2 \geq 0$, and $\tilde T_1(x, y) \geq 0$ for $\tilde \mu_0$-a.e. $(x, y) \in \R^2$.
	Equation \cref{eq:xyzcnotzero} with $w_0 \to 0$ gives
	\[
	\exp_e(-(\tilde T_1(x, y)- x) \diff u + (\tilde T_2(x, y) - y) \diff v) = (\tilde T_1(x, y) - x, \tilde T_2(x, y) - y, 0),
	\]
	and thus $\mathsf{p} \circ T = \tilde T$.
	Left-invariance of the time-separation function $\tau$ in $\H$ and \cref{eq:tausamez} yields
	\begin{align*}
		\tau((x, y, z), T(x, y, z)) ={}& \tau(e, \exp_e(-(\tilde T_1(x, y)- x) \diff u + (\tilde T_2(x, y) - y) \diff v))\\
		 ={}& \sqrt{(\tilde T_1(x, y)- x)^2 - (\tilde T_2(x, y) - y)^2} = \tilde \tau((x, y), \tilde T(x, y)).
	\end{align*}
	Since $\mathsf{p}_{\sharp} \mu_0 = \tilde \mu_0$, it follows that $T$ is a (forward) transport map from $\mu_0$ to $\mu_1$ and that $M^+(T) = \tilde M^+(\tilde T)$. If $S : \H \to \H$ is any (forward) transport map from $\mu_0$ to $\mu_1$, then it is easily seen from \cref{eq:J+} that $\mathsf{p} \circ S$ is a (forward) transport map from $\tilde \mu_0$ to $\tilde \mu_1$. From \cref{eq:tauleqptau}, we have that $M^+(S) \leq \tilde M^+(\mathsf{p} \circ S)$. Therefore, it holds that

	\[
	M^+(T) = \tilde M^+(\tilde T) \geq \tilde M^+(\mathsf{p} \circ S) \geq M^+(S),
	\]
	which shows that $T$ is optimal.
\end{proof}

Statements analogous to \cref{prop:projectionToptimal} may also be proven for \cref{eq:LorentzMonge-} and \cref{eq:LorentzKantorovich} in a similar fashion. Let us recall that the right-translation of the Heisenberg group from a given point $q_0 \in \H$ is the map $R_{q_0} : \H \to \H$ given by $R_{q_0}(q) = q \cdot q_0$. The sub-Riemannian version of the previous result can be found in \cite[Example 2.2.3]{Juillet2008}. A corresponding analogue of the next result appears in \cite[Example 5.7]{AmbrosioRigot2004} (see also the proof of \cite[Lemma 4.13]{AmbrosioRigot2004}).

\begin{theorem}
	Given $\mu_0 \in \mathcal{P}_{\mathrm{c}}(\H)$ with $\mu_0 \ll \mathcal{L}^3$, the right-translation $R_{q_0}$ from the point $q_0 = (x_0, y_0, z_0) \in \H$ is a (forward) optimal transport map from $\mu_0$ to $(R_{q_0})_{\sharp} \mu_0$ if and only if $z_0 = 0$ and $(x + x_0, y + y_0) \in \tilde I^{+}(x, y)$ for $\mu_0$-a.e. $(x, y) \in \R^2$.
\end{theorem}

\begin{proof}
	If $z_0 = 0$, then
	\begin{align*}
		R_{q_0}(q) ={}& (x + x_0, y + y_0, z) = (x, y, z) \cdot \exp_e(- x_0 \diff u + y_0 \diff v) \\
		={}& (x, y, z) \cdot \exp_e(-(\tilde T_1(x, y) - x) \diff u + (\tilde T_2(x, y) - y) \diff v),
	\end{align*}
	where $\tilde T(x, y) = (\tilde T_1(x, y), \tilde T_2(x, y)) := (x + x_0, y + y_0)$. Assuming $(x + x_0, y + y_0) \in \tilde I^{+}(x, y)$ for $\mu_0$-a.e. $(x, y) \in \R^2$ yields that $\tilde T$ is a (forward) optimal transport map from $\tilde \mu_0 = \mathsf{p}_\sharp \mu_0$ to $\tilde T_\sharp \tilde \mu_0$. By \cref{prop:projectionToptimal}, we conclude that $R_{q_0}$ is a (forward) optimal transport map from $\mu_0$ to $(R_{q_0})_\sharp \mu_0$.

	Suppose now that $R_{q_0}$ is a (forward) optimal transport map. We immediately obtain $(x + x_0, y + y_0) \in \tilde I^{+}(x, y)$ for $\mu_0$-a.e. $(x, y) \in \R^2$ by simply looking at \cref{eq:I+} and \cref{eq:tildeI+}. Brenier's Theorem implies that $R_{q_0}(q)$ can be written as \cref{eq:brenierexponential}
	for some $c_p$-concave map $\phi : A_1 \to \R$ given by \cref{thm:Brenier}. Since only the differential of $\phi$ is involved in \cref{thm:Brenier}, we assume, without loss of generality, that $\phi(q') = 0$ for some $q' = (x', y', z') \in A_1$. We deduce from \cref{remark:expqtoexpe} and from the left-invariance of the Hamiltonian $H$ that
	\begin{equation}
		\label{eq:righttranslationq0}
			q_0 = \exp_e\left(- (L_q)_e^*[\Diff_{q} \phi] \bigg/ \left(\sqrt{2H(e,(L_q)_e^*[\Diff_{q} \phi])}\right)^{\tfrac{p - 2}{p-1}}\right), \text{ for $\mu_0$-a.e. } q \in \H.
	\end{equation}
	\cref{thm:geodesicspace} gives a unique $\lambda_0 \in C_{e}(\H) \cap H^{-1}(1/2)$ such that $-(L_q)_e^*[\Diff_{q} \phi] = \lambda_0$ for $\mu_0$-a.e. $q \in \H$. Because $\phi$ is absolutely continuous, a computation in coordinates then implies
	\[
	\phi(x, y, z) = - \left(h_X(\lambda_0) + \frac{1}{2} h_Z(\lambda_0)\right) (x - x') - \left(h_Y(\lambda_0) - \frac{1}{2} h_Z(\lambda_0)\right) (y - y') - h_Z(\lambda_0) (z - z')
	\]
	We now use the fact that $\phi$ belongs to $\mathrm{BV}_{\mathrm{loc}, H}^2(A_1)$, which we know from \cref{lemma:locallysemiconcaveBVH}, and thus is twice differentiable $\mu_0$-almost everywhere. The result proven in \cite[Theorem 2.2]{Ambrosio2003} ensures that we have
	\[
	Z \phi(q) = [X, Y] \phi(q) = XY\phi(q) - YX\phi(q) = 0,
	\]
	for $\mu_0$-a.e. every $q \in \H$.
	Therefore, we have $h_Z(\lambda_0) = 0$, $\phi(x, y, z) = - h_X(\lambda_0) (x - x') - h_Y(\lambda_0) (y - y')$, as well as $\Diff_{q} \phi = - h_X(\lambda_0) \diff x - h_Y(\lambda_0) \diff y$, and thus $z_0 = 0$ thanks to \cref{eq:righttranslationq0}.
\end{proof}



 \begin{spacing}{0.5}
 \printbibliography[heading=bibintoc]

@Article{Figalli2010,
  author     = {Figalli, Alessio and Rifford, Ludovic},
  journal    = {Geom. Funct. Anal.},
  title      = {Mass transportation on sub-{R}iemannian manifolds},
  year       = {2010},
  issn       = {1016-443X,1420-8970},
  number     = {1},
  pages      = {124--159},
  volume     = {20},
  doi        = {10.1007/s00039-010-0053-z},
  fjournal   = {Geometric and Functional Analysis},
  mrclass    = {49Q20 (53C17)},
  mrnumber   = {2647137},
  mrreviewer = {Luca\ Granieri},
}

@Book{Santambrogio2015,
  author     = {Santambrogio, Filippo},
  publisher  = {Birkh\"auser/Springer, Cham},
  title      = {Optimal transport for applied mathematicians},
  year       = {2015},
  note       = {Calculus of variations, PDEs, and modeling},
  series     = {Progress in Nonlinear Differential Equations and their Applications},
  volume     = {87},
  doi        = {10.1007/978-3-319-20828-2},
  mrclass    = {49-02 (35J96 49J45 49M29 58E50 90C05 90C48 91B02)},
  mrnumber   = {3409718},
  mrreviewer = {Luigi\ De Pascale},
  pages      = {xxvii+353},
}

@Article{Brenier1991,
  author     = {Brenier, Yann},
  journal    = {Comm. Pure Appl. Math.},
  title      = {Polar factorization and monotone rearrangement of vector-valued functions},
  year       = {1991},
  issn       = {0010-3640,1097-0312},
  number     = {4},
  pages      = {375--417},
  volume     = {44},
  doi        = {10.1002/cpa.3160440402},
  fjournal   = {Communications on Pure and Applied Mathematics},
  mrclass    = {46E40 (35Q99 46E99 49Q99)},
  mrnumber   = {1100809},
  mrreviewer = {Robert\ McOwen},
}

@Article{Mondino2023,
  author     = {Mondino, Andrea and Suhr, Stefan},
  journal    = {J. Eur. Math. Soc. (JEMS)},
  title      = {An optimal transport formulation of the {E}instein equations of general relativity},
  year       = {2023},
  issn       = {1435-9855,1435-9863},
  number     = {3},
  pages      = {933--994},
  volume     = {25},
  doi        = {10.4171/jems/1188},
  fjournal   = {Journal of the European Mathematical Society (JEMS)},
  mrclass    = {53C80 (49Q22 53C50 82C35 83C05)},
  mrnumber   = {4577957},
  mrreviewer = {Michael\ Kunzinger},
}

@Book{Villani2009,
  author     = {Villani, C\'edric},
  publisher  = {Springer-Verlag, Berlin},
  title      = {Optimal transport},
  year       = {2009},
  isbn       = {978-3-540-71049-3},
  note       = {Old and new},
  series     = {Grundlehren der mathematischen Wissenschaften [Fundamental Principles of Mathematical Sciences]},
  volume     = {338},
  doi        = {10.1007/978-3-540-71050-9},
  mrclass    = {49-02 (28A75 37J50 49Q20 53C23 58E30)},
  mrnumber   = {2459454},
  mrreviewer = {Dario\ Cordero-Erausquin},
  pages      = {xxii+973},
}

@Book{Capogna2007,
  author     = {Capogna, Luca and Danielli, Donatella and Pauls, Scott D. and Tyson, Jeremy T.},
  publisher  = {Birkh\"auser Verlag, Basel},
  title      = {An introduction to the {H}eisenberg group and the sub-{R}iemannian isoperimetric problem},
  year       = {2007},
  series     = {Progress in Mathematics},
  volume     = {259},
  doi        = {10.1007/978-3-7643-8133-2},
  mrclass    = {53C17 (22E30 30C65 32T27 32V15 49Q15)},
  mrnumber   = {2312336},
  mrreviewer = {Piotr\ Haj\l asz},
  pages      = {xvi+223},
}

@Article{Strichartz1986,
  author     = {Strichartz, Robert S.},
  journal    = {J. Differential Geom.},
  title      = {Sub-{R}iemannian geometry},
  year       = {1986},
  issn       = {0022-040X,1945-743X},
  number     = {2},
  pages      = {221--263},
  volume     = {24},
  fjournal   = {Journal of Differential Geometry},
  mrclass    = {53C20 (53A40 53C21 53C22 58G30)},
  mrnumber   = {862049},
  mrreviewer = {Karsten\ Grove},
  url        = {http://projecteuclid.org/euclid.jdg/1214440436},
}

@Book{Montgomery2002,
  author     = {Montgomery, Richard},
  publisher  = {American Mathematical Society},
  title      = {A Tour of Subriemannian Geometries, Their Geodesics and Applications},
  year       = {2002},
  isbn       = {0-8218-1391-9},
  month      = aug,
  series     = {Mathematical Surveys and Monographs},
  volume     = {91},
  doi        = {10.1090/surv/091},
  issn       = {2331-7159},
  journal    = {Mathematical Surveys and Monographs},
  mrclass    = {53C17 (37J99 53C60 58E10 70G45 70H05)},
  mrnumber   = {1867362},
  mrreviewer = {Andrey\ V.\ Sarychev},
  pages      = {xx+259},
}

@PhdThesis{Juillet2008,
  author = {Juillet, Nicolas},
  school = {Universit{\'e} Joseph-Fourier-Grenoble I; Rheinische Friedrich-Wilhelms-Universit\"{a}t Bonn},
  title  = {Optimal transport and geometric analysis in Heisenberg groups},
  year   = {2008},
  url    = {https://theses.hal.science/tel-00345301/},
}

@Article{Grochowski2006,
  author     = {Grochowski, Marek},
  journal    = {J. Dyn. Control Syst.},
  title      = {Reachable sets for the {H}eisenberg sub-{L}orentzian structure on {${\mathbb R}^3$}. {A}n estimate for the distance function},
  year       = {2006},
  issn       = {1079-2724,1573-8698},
  number     = {2},
  pages      = {145--160},
  volume     = {12},
  doi        = {10.1007/s10450-006-0378-y},
  fjournal   = {Journal of Dynamical and Control Systems},
  mrclass    = {53C17 (53C50 93B03)},
  mrnumber   = {2222194},
  mrreviewer = {Yilong\ Ni},
}

@InCollection{Grochowski2004,
  author     = {Grochowski, Marek},
  booktitle  = {Geometric singularity theory},
  publisher  = {Polish Acad. Sci. Inst. Math., Warsaw},
  title      = {On the {H}eisenberg sub-{L}orentzian metric on {$\mathbb R^3$}},
  year       = {2004},
  pages      = {57--65},
  series     = {Banach Center Publ.},
  volume     = {65},
  doi        = {10.4064/bc65-0-4},
  mrclass    = {53C50 (53C17)},
  mrnumber   = {2104337},
  mrreviewer = {Paul\ E.\ Ehrlich},
}

@Article{Grochowski2002,
  author     = {Grochowski, Marek},
  journal    = {Bull. Polish Acad. Sci. Math.},
  title      = {Geodesics in the sub-{L}orentzian geometry},
  year       = {2002},
  issn       = {0239-7269},
  number     = {2},
  pages      = {161--178},
  volume     = {50},
  fjournal   = {Polish Academy of Sciences. Bulletin. Mathematics},
  mrclass    = {53C17 (53C50)},
  mrnumber   = {1923381},
  mrreviewer = {Paul\ E.\ Ehrlich},
}

@Article{Barilari2019,
  author     = {Barilari, Davide and Rizzi, Luca},
  journal    = {Invent. Math.},
  title      = {Sub-{R}iemannian interpolation inequalities},
  year       = {2019},
  issn       = {0020-9910,1432-1297},
  number     = {3},
  pages      = {977--1038},
  volume     = {215},
  doi        = {10.1007/s00222-018-0840-y},
  fjournal   = {Inventiones Mathematicae},
  mrclass    = {53C17 (49J15 49Q20)},
  mrnumber   = {3935035},
  mrreviewer = {Emmanuel\ Tr\'elat},
}

@Article{Huang2012,
  author     = {Huang, Tiren and Yang, Xiaoping},
  journal    = {J. Dyn. Control Syst.},
  title      = {Geodesics in the {H}eisenberg group {$H^n$} with a {L}orentzian metric},
  year       = {2012},
  issn       = {1079-2724,1573-8698},
  number     = {4},
  pages      = {479--498},
  volume     = {18},
  doi        = {10.1007/s10883-012-9156-1},
  fjournal   = {Journal of Dynamical and Control Systems},
  mrclass    = {53C50 (58E10)},
  mrnumber   = {2980535},
  mrreviewer = {Davide\ Vittone},
}

@Article{Ambrosio2003,
  author     = {Ambrosio, Luigi and Magnani, Valentino},
  journal    = {Math. Z.},
  title      = {Weak differentiability of {BV} functions on stratified groups},
  year       = {2003},
  issn       = {0025-5874,1432-1823},
  number     = {1},
  pages      = {123--153},
  volume     = {245},
  doi        = {10.1007/s00209-003-0530-2},
  fjournal   = {Mathematische Zeitschrift},
  mrclass    = {43A80 (22E30 26B30 42B35 46E35)},
  mrnumber   = {2023957},
  mrreviewer = {Gerald\ B.\ Folland},
}

@Article{McCann2020,
  author   = {McCann, Robert J.},
  journal  = {Camb. J. Math.},
  title    = {Displacement convexity of {B}oltzmann's entropy characterizes the strong energy condition from general relativity},
  year     = {2020},
  issn     = {2168-0930,2168-0949},
  number   = {3},
  pages    = {609--681},
  volume   = {8},
  doi      = {10.4310/CJM.2020.v8.n3.a4},
  fjournal = {Cambridge Journal of Mathematics},
  mrclass  = {53C50 (49Q22 53C21 58Z05 82C35 83C99)},
  mrnumber = {4192570},
}

@Book{Cannarsa2004,
  author     = {Cannarsa, Piermarco and Sinestrari, Carlo},
  publisher  = {Birkhäuser},
  title      = {Semiconcave functions, Hamilton-Jacobi equations, and optimal control},
  year       = {2004},
  address    = {Boston},
  isbn       = {0-8176-4084-3},
  note       = {Includes bibliographical references and index},
  number     = {58},
  series     = {Progress in nonlinear differential equations and their applications},
  volume     = {58},
  mrclass    = {49-02 (35F20 49K20 49L20)},
  mrnumber   = {2041617},
  mrreviewer = {Pierre\ Cardaliaguet},
  pages      = {xiv+304},
  pagetotal  = {304},
  ppn_gvk    = {378498061},
}

@book {burago-burago-sergei2001,
    AUTHOR = {Burago, Dmitri and Burago, Yuri and Ivanov, Sergei},
     TITLE = {A course in metric geometry},
    SERIES = {Graduate Studies in Mathematics},
    VOLUME = {33},
 PUBLISHER = {American Mathematical Society, Providence, RI},
      YEAR = {2001},
     PAGES = {xiv+415},
      ISBN = {0-8218-2129-6},
   MRCLASS = {53C23},
  MRNUMBER = {1835418},
MRREVIEWER = {Mario\ Bonk},
       DOI = {10.1090/gsm/033},
       URL = {https://doi.org/10.1090/gsm/033},
}

@misc{piotr2011,
Author = {Piotr T. Chruściel},
Title = {Elements of causality theory},
Year = {2011},
Eprint = {arXiv:1110.6706},
}

@article {felixgluing,
    AUTHOR = {Rott, Felix},
     TITLE = {Gluing of {L}orentzian length spaces and the causal ladder},
   JOURNAL = {Classical Quantum Gravity},
  FJOURNAL = {Classical and Quantum Gravity},
    VOLUME = {40},
      YEAR = {2023},
    NUMBER = {17},
     PAGES = {Paper No. 175002, 28},
      ISSN = {0264-9381,1361-6382},
   MRCLASS = {53C50 (51P05 53C21)},
  MRNUMBER = {4622309},
}

@article {kuzingersamann2018,
    AUTHOR = {Kunzinger, Michael and S\"amann, Clemens},
     TITLE = {Lorentzian length spaces},
   JOURNAL = {Ann. Global Anal. Geom.},
  FJOURNAL = {Annals of Global Analysis and Geometry},
    VOLUME = {54},
      YEAR = {2018},
    NUMBER = {3},
     PAGES = {399--447},
      ISSN = {0232-704X,1572-9060},
   MRCLASS = {53C23 (53B30 53C50 53C80)},
  MRNUMBER = {3867652},
MRREVIEWER = {Benjam\'in\ Olea},
       DOI = {10.1007/s10455-018-9633-1},
       URL = {https://doi.org/10.1007/s10455-018-9633-1},
}

@Article{AmbrosioRigot2004,
  author     = {Ambrosio, Luigi and Rigot, Séverine},
  journal    = {Journal of Functional Analysis},
  title      = {Optimal mass transportation in the Heisenberg group},
  year       = {2004},
  issn       = {0022-1236,1096-0783},
  month      = mar,
  number     = {2},
  pages      = {261--301},
  volume     = {208},
  doi        = {10.1016/S0022-1236(03)00019-3},
  fjournal   = {Journal of Functional Analysis},
  mrclass    = {49Q20 (35H10 43A80 53C17)},
  mrnumber   = {2035027},
  mrreviewer = {Enrico\ Valdinoci},
  publisher  = {Elsevier BV},
}

@article {BraunOhta2024,
    AUTHOR = {Braun, Mathias and Ohta, Shin-ichi},
     TITLE = {Optimal transport and timelike lower {R}icci curvature bounds
              on {F}insler spacetimes},
   JOURNAL = {Trans. Amer. Math. Soc.},
  FJOURNAL = {Transactions of the American Mathematical Society},
    VOLUME = {377},
      YEAR = {2024},
    NUMBER = {5},
     PAGES = {3529--3576},
      ISSN = {0002-9947,1088-6850},
   MRCLASS = {53C60 (53C50 58E10 83C75)},
  MRNUMBER = {4744787},
       DOI = {10.1090/tran/9126},
       URL = {https://doi.org/10.1090/tran/9126},
}

@book{Agrachev2004,
  title = {Control Theory from the Geometric Viewpoint},
  ISBN = {9783662064047},
  ISSN = {0938-0396},
  url = {http://dx.doi.org/10.1007/978-3-662-06404-7},
  DOI = {10.1007/978-3-662-06404-7},
  journal = {Encyclopaedia of Mathematical Sciences},
  publisher = {Springer Berlin Heidelberg},
  author = {Agrachev,  Andrei A. and Sachkov,  Yuri L.},
  year = {2004}
}

@article {sachkov2022sublorentzian,
    AUTHOR = {Sachkov, Yuri L. and Sachkova, Elena F.},
     TITLE = {Sub-{L}orentzian distance and spheres on the {H}eisenberg
              group},
   JOURNAL = {J. Dyn. Control Syst.},
  FJOURNAL = {Journal of Dynamical and Control Systems},
    VOLUME = {29},
      YEAR = {2023},
    NUMBER = {3},
     PAGES = {1129--1159},
      ISSN = {1079-2724,1573-8698},
   MRCLASS = {53C17 (49K15 53C50)},
  MRNUMBER = {4645081},
       DOI = {10.1007/s10883-023-09652-2},
       URL = {https://doi.org/10.1007/s10883-023-09652-2},
}

@book {comprehensive2020,
	AUTHOR = {Agrachev, Andrei and Barilari, Davide and Boscain, Ugo},
	title={A Comprehensive Introduction to Sub-Riemannian Geometry},
	SERIES = {Cambridge Studies in Advanced Mathematics},
	VOLUME = {181}, 
	PUBLISHER = {Cambridge University Press, Cambridge},
	YEAR = {2020},
	ISBN = {978-1-108-47635-5},
	MRCLASS = {53C17},
	MRNUMBER = {3971262},
	MRREVIEWER = {Luca Rizzi},
	DOI={10.1017/9781108677325}, 
}

@book{rifford2014book,
    year = {2014},
    author = {Rifford, Ludovic},
    booktitle = {Sub-Riemannian Geometry and Optimal Transport},
    isbn = {9783319048048},
    publisher = {Springer},
    series = {Springer Briefs in Mathematics},
    title = {Sub-Riemannian Geometry and Optimal Transport },
    DOI = {10.1007/978-3-319-04804-8}
}

@article{CavallettiMondino2022,
   title={A review of Lorentzian synthetic theory of timelike Ricci curvature bounds},
   volume={54},
   ISSN={1572-9532},
   url={http://dx.doi.org/10.1007/s10714-022-03004-4},
   DOI={10.1007/s10714-022-03004-4},
   number={11},
   journal={General Relativity and Gravitation},
   publisher={Springer Science and Business Media LLC},
   author={Cavalletti, Fabio and Mondino, Andrea},
   year={2022},
   month=oct }

@Article{CavallettiMondino2023,
  author   = {Cavalletti, Fabio and Mondino, Andrea},
  journal  = {Camb. J. Math.},
  title    = {Optimal transport in {L}orentzian synthetic spaces, synthetic timelike {R}icci curvature lower bounds and applications},
  year     = {2024},
  issn     = {2168-0930,2168-0949},
  number   = {2},
  pages    = {417--534},
  volume   = {12},
  fjournal = {Cambridge Journal of Mathematics},
  mrclass  = {53C23 (49Q22 53C50 53C80 83C75)},
  mrnumber = {4779676},
}

@Article{braun2022good,
  author     = {Braun, Mathias},
  journal    = {Nonlinear Analysis},
  title      = {Good geodesics satisfying the timelike curvature-dimension condition},
  year       = {2023},
  issn       = {0362-546X,1873-5215},
  month      = apr,
  pages      = {Paper No. 113205, 30},
  volume     = {229},
  doi        = {10.1016/j.na.2022.113205},
  fjournal   = {Nonlinear Analysis. Theory, Methods \& Applications. An International Multidisciplinary Journal},
  mrclass    = {49Q20 (49J52 53C23 53C50 58E10)},
  mrnumber   = {4528587},
  mrreviewer = {David\ Tewodrose},
  publisher  = {Elsevier BV},
}

@article{mccann2001,
        author = {McCann, Robert J.},
        year = {2001},
        month = {08},
        pages = {589-608},
        title = {Polar factorization of maps on Riemannian manifolds},
        volume = {11},
        journal = {Geometric and Functional Analysis},
        doi = {10.1007/PL00001679}
}
 \end{spacing}

\end{document}